\documentclass[11pt]{amsart}

\pdfpageattr{/Group <</S /Transparency /I true /CS /DeviceRGB>>}

\usepackage{amsfonts,amsmath,amssymb,amsthm}
\usepackage[textwidth=5in, textheight=8.37in, twoside=false]{geometry}
\usepackage[english]{babel}
\usepackage{tikz}
\usepackage{tkz-euclide}
\usetikzlibrary{decorations.markings,shapes}
\usepackage{babel}
\usepackage{graphicx}
\usepackage{subcaption}
\usepackage{textcomp}
\usepackage{bbm}
\usepackage[T1]{fontenc}
\usepackage[charter]{mathdesign}
\usepackage{paralist}
\usepackage{enumitem}
\usepackage[colorinlistoftodos, bordercolor=orange, backgroundcolor=orange!20, linecolor=orange, textsize=scriptsize]{todonotes}
\usepackage{aliascnt}
\usepackage{xcolor}
\usepackage{csquotes}
\usepackage{microtype}
\usepackage{subcaption}
\usepackage{xspace}

\DeclareMathOperator{\Triang}{Triang}
\newcommand{\Riem}{\mathcal{R}}
\newcommand{\Sgn}{\mathcal{S}_{g,n}}
\newcommand{\Teich}{\mathcal{T}_{g,n}}
\newcommand{\decTeich}{\widetilde{\mathcal{T}}_{g,n}}
\newcommand\HH{{\mathbb H}}

\DeclareMathOperator{\SO}{SO^+(2,1)}
\newcommand{\HS}{U}
\DeclareMathOperator{\PSL2R}{PSL_2(\mathbb{R})}
\newcommand{\hemi}{N}
\newcommand{\Klein}{K}
\newcommand{\Lplus}{L^+}

\newcommand{\B}{\mathcal{B}}
\newcommand{\C}{\mathcal{C}}
\newcommand{\DD}{D}
\newcommand{\triang}{T}

\DeclareMathOperator{\secpoly}{\Sigma\operatorname{-poly}}
\DeclareMathOperator{\secfan}{\Sigma\operatorname{-fan}}

\newcommand\R{\mathbb{R}}

\newcommand{\Rpos}{\mathbb{R}_{\geq 0}}
\newcommand{\scalp}[1]{\langle #1 \rangle}

\DeclareMathOperator{\vol}{vol}
\DeclareMathOperator{\conv}{conv}

\newcommand{\Cone}{C}
\newcommand{\itr}{\Delta}
\newcommand{\hl}{\alpha}
\newcommand\polymake{\texttt{polymake}\xspace}
\newcommand\gfan{\texttt{Gfan}\xspace}
\newcommand{\PP}{\mathcal{P}}

\theoremstyle{plain}
\newtheorem{theorem}{Theorem}[section]
\newtheorem{proposition}[theorem]{Proposition}
\newtheorem{lemma}[theorem]{Lemma}
\newtheorem{corollary}[theorem]{Corollary}
\newtheorem{thmdef}[theorem]{Theorem and Definition}
\theoremstyle{definition}
\newtheorem{definition}[theorem]{Definition}
\newtheorem{example}[theorem]{Example}
\newtheorem{question}[theorem]{Question}
\newtheorem{remark}[theorem]{Remark}

\definecolor{blue1}{rgb}{0.3,0.65,1}
\definecolor{blue2text}{rgb}{0.3,0.3,1}
\definecolor{blue2}{rgb}{0.2,0.3,1}
\definecolor{red1}{rgb}{1,0.75,0.75}
\definecolor{red2}{rgb}{1,0.25,0.35}
\definecolor{red2text}{rgb}{1,0.16,0.25}

\include{pictures/tikz_examples}

\usepackage{hyperref}
\definecolor{links}{rgb}{.2,.1,.5}
\definecolor{cites}{rgb}{.5,.1,.2}
\hypersetup{colorlinks,linkcolor=links,citecolor=cites,urlcolor=links}

\AtBeginDocument{
	\setlength{\leftmargini}{2em}
	\setlength{\leftmarginii}{2em}
	\setlength{\leftmarginiii}{2em}
	\setlength{\leftmarginiv}{2em}
}

\setenumerate{noitemsep,topsep=0.5em,parsep=0.5em,partopsep=0.5em}

\title[Secondary Polyhedra of Riemann Surfaces]{Secondary Fans and Secondary Polyhedra\\ of Punctured Riemann Surfaces}

\author[M. Joswig, R. L\"owe \and B. Springborn]{Michael Joswig, Robert L\"owe \and Boris Springborn}

\address{Technische Universit\"at  Berlin \\ Institut f{\"u}r Mathematik \\ Str.\ des 17. Juni 136 \\ 10623 Berlin, Germany}

\thanks{%
  This work is supported by DFG via SFB-TRR 109: ``Discretization in Geometry and Dynamics''.
  Further support for M. Joswig from Einstein Foundation Berlin (within the framework of Matheon) and DFG via SFB-TRR 195: ``Symbolic Tools in Mathematics and their Application''.
}



\begin{document}

\begin{abstract}
  A famous construction of Gel'fand, Kapranov and Zelevinsky associates to each finite point configuration $A \subset \R^d$ a polyhedral fan, which stratifies the space of weight vectors by the combinatorial types of regular subdivisions of $A$. 
  That fan arises as the normal fan of a convex polytope.
  In a completely analogous way we associate to each hyperbolic Riemann surface $\Riem$ with punctures a polyhedral fan.
  Its cones correspond to the ideal cell decompositions of $\Riem$ that occur as the horocyclic Delaunay decompositions which arise via the convex hull construction of Epstein and Penner.
  Similar to the classical case, this \emph{secondary fan} of $\Riem$ turns out to be the normal fan of a convex polyhedron, the \emph{secondary polyhedron} of $\Riem$.
\end{abstract}

\maketitle


\section{Introduction}
\noindent
Our goal is to employ techniques from geometric combinatorics to further the understanding of the space of ideal Delaunay decompositions of a punctured Riemann surface.
We believe that this is relevant to researchers in hyperbolic geometry.
Combinatorialists may benefit from seeing how far their methods carry.

A punctured Riemann surface $\Riem$ (of genus $g$ with $n$ punctures)
is $\Sgn$, the closed oriented surface of genus $g$ with $n$ punctures,
equipped with a complete hyperbolic metric of finite
area. Throughout this article we assume that the Euler characteristic
$2-2g-n$ is negative and $n\geq1$. The punctures of $\Sgn$ correspond to cusps of
$\Riem$ (see Figure~\ref{fig:cusp}).  The space of complete hyperbolic
metrics with finite area on $\Sgn$, up to isotopy, is known as the
\emph{Teichm\"uller space}~$\Teich$.
Penner~\cite{penner1} suggested to equip a punctured Riemann surface
with an additional choice of horocycles, one for each
puncture.  This is known as a \emph{decoration} of $\Riem$.  Via their
lengths, such a choice of horocycles can be described by a vector of
$n$ positive real numbers, the \emph{weight vector} of the decoration.
In this way we obtain the \emph{decorated Teichm\"uller space}
$\decTeich$, which is a trivial $\R_{>0}^n$-bundle over~$\Teich$.
Akiyoshi \cite{akiyoshi} used partial decorations with weight
vectors in $\R_{\geq0}^n\setminus\{0\}$.

Epstein and Penner \cite{epstein_penner, penner1} employed a
\emph{convex hull construction} to show that a decoration of $\Riem$
determines an ideal cell decomposition of $\Riem$. This construction
involves the hyperboloid model of the hyperbolic plane $H$ and the
representation of $\Riem$ as a quotient $H/\Gamma$ with respect to a
group $\Gamma$ of hyperbolic isometries. Distinguishing decorated
Riemann surfaces by the induced decomposition of $\Sgn$, one obtains a
cell decomposition of the decorated Teichm{\"u}ller space
$\decTeich$~\cite{penner1, pennerbook}. In this article we focus on
individual fibers of $\decTeich$, i.e., we consider fixed Riemann
surfaces with variable decoration. If the Riemann surface is fixed,
only a finite number of ideal cell decompositions occur as result of
the the convex hull construction~\cite{akiyoshi}.

Our first observation is that for a fixed Riemann surface $\Riem$, the
weight vectors which induce the same ideal cell decomposition of
$\Riem$ form a relatively open polyhedral cone
(Theorem~\ref{thm:cones}).  Moreover, since these \emph{secondary
  cones} meet face-to-face, we obtain a polyhedral fan, the
\emph{secondary fan} of $\Riem$ (cf.\ Definition~\ref{def:sfan}).
In particular, this
shows that the refinement poset of ideal Delaunay decompositions is a
lattice with the minimal and maximal elements removed: If the set of
ideal Delaunay decompositions contains a common refinement for two of
its elements, then it contains a unique coarsest common refinement,
and if it contains a common coarsening, then it contains a unique
finest common coarsening.

This is very similar to the secondary cones and secondary fans of
point configurations in $\R^{d}$, 
which were introduced by Gel'fand, Kapranov and Zelevinsky
\cite{gkz}, and which are the fundamental building blocks of a theory
with numerous applications to combinatorics, optimization, algebra
and other parts of mathematics; see the monograph of De Loera, Rambau and Santos
\cite{bibel}.  A key theorem says that the secondary fan of a
Euclidean point configuration arises as the normal fan of a convex
polytope, the \emph{secondary polytope} of the point
configuration. Our main result (Theorem~\ref{thm:normalfan}) is a
complete analog for punctured Riemann surfaces: The secondary fan
of $\Riem$ is the normal fan of a \emph{secondary polyhedron} of
$\Riem$ (cf.\ Definition~\ref{def:secpoly}).

Despite the close analogy, there are some notable differences between
the classical theory of secondary polytopes and our version for
punctured Riemann surfaces. First of all, for punctured Riemann
surfaces only non-negative weights are allowed because they are the
lengths of the horocycles at the punctures. This prevents our
secondary polyhedra from being bounded. While all vertices of a
secondary polytope of a point configuration correspond to
triangulations, the vertices of a secondary polyhedron of a punctured
Riemann surface correspond to coarsest Delaunay decompositions, which
are not necessarily ideal triangulations. Moreover, while a point configuration
in $\R^{d}$ determines a unique secondary polytope, our construction
associates a unique secondary polyhedron to each pair $(\Riem,x)$
consisting of a punctured Riemann surface $\Riem$ and a point
$x\in\Riem$ (cf.\ Section~\ref{sec:poly}).  Thus, our
construction yields a polyhedron bundle over the Riemann surface
$\Riem$. This is reminiscent of the notion of fiber
polytopes~\cite{BilleraSturmfels:1992}, but here the base space is a
punctured Riemann surface instead of a polytope.

This research was originally motivated by a recent variational method
to construct ideal hyperbolic polyhedra with prescribed intrinsic
metric, or equivalently, to compute discrete uniformizations of
piecewise Euclidean surfaces~\cite{springb_uniform}.
The complexity of this method depends on the complexity of the flip
algorithm for the Epstein--Penner convex hull construction~\cite{Weeks:1993,TillmannWong:2016}.
We expect that the correspondence between ideal Delaunay decompositions of a punctured Riemann surface and faces of a secondary polyhedron will shed further light on that method.

The paper is organized in a way to make it accessible for audiences from both geometry and combinatorics.
To describe the full setup thus requires an unusually long preparation. Experts in both fields might want to start with the new results in Section~\ref{sec:fan} right away.
Frequent references to the introductory sections are meant as an aid for picking up our notation.
In Section~\ref{sec:gkz} we begin with reviewing the classical GKZ construction of secondary polytopes of point configurations in $\R^{d}$.
This will subsequently allow to point out the similarities and the differences between the classical theory and our version for punctured Riemann surfaces.
In Section~\ref{sec:hyperbolic} we recall basic facts from hyperbolic geometry.
This is mainly to introduce our notation, but also for describing explicitly how to translate between various models of the hyperbolic plane.
This is important as the proofs of our main results require to switch freely between several models.
After a brief review of punctured Riemann surfaces and Penner's coordinates on decorated Teichm{\"u}ller spaces in Section~\ref{sec:riemann}, we finally define the secondary fan of a punctured Riemann surface in Section~\ref{sec:fan}.
Ideal cell decompositions of a Riemann surface with $n$ punctures correspond to secondary cones in $\R^{n}$, which form the secondary fan.
The construction of secondary polyhedra and our main result, Theorem~\ref{thm:normalfan}, are the topic of Section~\ref{sec:poly}.
Sections~\ref{sec:fan} and~\ref{sec:poly} are illustrated with many explicit examples.
The latter have been obtained via an implementation in \polymake \cite{polymake:2000}, and the method is briefly explained in Section~\ref{sec:computing}.
We close the paper with remarks on possible generalizations and open questions in Section~\ref{sec:concluding}.

For helpful discussions we are indebted to Stephan Tillmann.

\section{The classical GKZ construction}\label{sec:gkz} 
\noindent
In this section, we quickly recall the classical constructions of
secondary fans and secondary polytopes for point configurations in
$\R^{d}$.  The ideas were developed by Gel'fand, Kapranov and
Zelevinsky \cite{gkz}; see also \cite[Chap.~5]{bibel}. In
Sections~\ref{sec:fan} and~\ref{sec:poly} we will
describe analogous constructions for punctured Riemann surfaces
instead of point configurations in $\R^{d}$.

Let $A\subset\R^{d}$ be a non-empty finite subset with $n=|A|$
elements and let $Q=\conv(A)$ be its convex hull. For simplicity we
assume that $A$ is not contained in a proper affine subspace of
$\R^{d}$, so that $Q$ is a $d$-dimensional polytope and in particular
$n\geq d+1$. A \emph{polytopal subdivision} of $(Q,A)$ is a
polytopal complex~$S$ whose carrier is $Q$ and whose vertex set is a
subset of $A$. Note that it is not required that all points in $A$ are
vertices of $S$. Only the vertices of $Q$ necessarily occur as
vertices of $S$. A \emph{triangulation of $(Q,A)$} is a polytopal
subdivision of $(Q,A)$ whose elements are simplices.

For any assignment $\omega\in\R^{A}$ of real numbers to elements of
$A$, let $g_{\omega}$ be the function
\begin{equation}
  \label{eq:gomega}
  \begin{gathered}
    g_{\omega}:Q\longrightarrow\R\\
    g_{\omega}(x)=\min_{h}(h(x)),
  \end{gathered}
\end{equation}
where the minimum is taken over all affine functions
\begin{equation}
  \label{eq:haff}
  h:\R^{d}\longrightarrow\R,\quad
  h(x)=h_{0}+h_{1}x_{1}+\ldots+h_{d}x_{d}
\end{equation}
satisfying 
\begin{equation}
  \label{eq:haffcond}
  h(a)\geq\omega_{a}\quad\text{for all}\quad a\in A.
\end{equation}
The function $g_{\omega}$ is a piecewise linear concave function. Its
graph
\begin{equation*}
  G(\omega) \ = \
  \big\{~(x,y)\in\R^{d}\times\R~\big|~x\in Q,~y=g_{\omega}(x)~\big\}
\end{equation*}
is the upper boundary of the convex hull of the point set
$\{(a,\omega_{a})\}_{a\in A}$ in $\R^{d}\times\R$. 

Let $D(\omega)$ be the polytopal subdivision of $(Q,A)$ that
contains a polytope $P\subseteq Q$ if and only if there is an affine
function~\eqref{eq:haff} satisfying~\eqref{eq:haffcond} such that 
\begin{equation*}
  P \ = \ \big\{~x\in Q~\big|~g_{\omega}(x)=h(x)~\big\}.
\end{equation*}
Vertical projection $\R^{d}\times\R\rightarrow\R^{d}$ maps the faces
of $G(\omega)$ bijectively onto the cells of $D(\omega)$.
A polytopal subdivision $S$ of $(Q,A)$ is called \emph{regular} if
$S=D(\omega)$ for some $\omega\in\R^{A}$. 
The \emph{secondary cone} $C(S)\subseteq\R^{A}$ of a polytopal
subdivision $S$ of $(Q,A)$ is defined by
\begin{equation}
  \label{eq:GKZcone}
  \begin{split}
    C(S)
    \ := \ &\big\{~\omega\in\R^{A}~\big|~g_{\omega}
    \text{ is affine on all cells of }S~\big\}\\
    = \ &\big\{~\omega\in\R^{A}~\big|~S\preceq D(\omega)~\big\},
  \end{split}
\end{equation}
where we write $S_{1}\preceq S_{2}$ if $S_{1}$ refines $S_{2}$, i.e.,
if every cell of $S_{1}$ is contained in a some cell of
$S_{2}$. The secondary cones are indeed polyhedral cones. To see this,
we use the functions $g_{T,\omega}$ defined below to derive linear
equations and inequalities describing the secondary cones. We will
also use the functions $g_{T,\omega}$ to define the secondary
polytopes.

For any triangulation $T$ of $(Q,A)$ and any $\omega\in\R^{A}$, let
\begin{equation}
  \label{eq:gTomega}
  g_{T,\omega}: Q\longrightarrow \R
\end{equation}
be the linear interpolation of $\omega$ with respect to $T$, i.e.,
the unique piecewise linear function that is affine on each simplex of
$T$ and satisfies 
\begin{equation*}
  g_{T,\omega}(a) \ = \ \omega_{a}\quad\text{for all}\quad a\in A \enspace. 
\end{equation*}
Then $g_{T,\omega}=g_{\omega}$ if and only if $g_{T,\omega}$ is
concave, so
\begin{equation}
  \label{eq:GKZcone2}
  C(T) \ = \ \big\{~\omega\in\R^{A}~\big|~g_{T,\omega}\text{ is concave}~\big\}.
\end{equation}
On each $d$-simplex $\sigma=[a_{1},\ldots,a_{d+1}]\in T$, the function
$g_{T,\omega}$ coincides with the affine function
\begin{equation}
  \label{eq:gtomegasigma}
  g_{T,\omega}^{\sigma}(x) \ = \ \frac{1}{d!\vol(\sigma)}
  \det
  \begin{pmatrix}
    a_{1}&\ldots&a_{d+1}& -x\\
    1 &\ldots& 1    & -1\\
    \omega_{a_{1}}&\ldots&\omega_{a_{d+1}}&0
  \end{pmatrix} \enspace ,
\end{equation}
where $\vol(\sigma)$ denotes the oriented volume of an oriented
$d$-simplex in $\R^{d}$, i.e.,
\begin{equation*}
  \vol([x_{1},\ldots,x_{d+1}]) \ = \ \frac{1}{d!}\det
  \begin{pmatrix}
    x_{1}&\ldots&x_{d+1}\\
    1 &\ldots& 1
  \end{pmatrix} \enspace .
\end{equation*}
The function $g_{T,\omega}$ is concave if and only if for every pair
\begin{equation*}
  [a_{1},\ldots,a_{d+1}],[a_{2},\ldots,a_{d+2}]\in T
\end{equation*}
of $d$-simplices sharing a $(d-1)$-face $[a_{2},\dots,a_{d+1}]$ we have
\begin{equation}
  \label{eq:gconvcond}
  g^{[a_{1},\ldots,a_{d+1}]}_{T,\omega}(a_{d+2}) \ \geq \ \omega_{a_{d+2}} \enspace.
\end{equation}
If the simplex $[a_{1},\ldots,a_{d+1}]$ is positively oriented, i.e., $\vol([a_{1},\ldots,a_{d+1}])>0$, then inequality~\eqref{eq:gconvcond} is equivalent to
\begin{equation}
  \label{eq:GKZconeineq}
  \det
  \begin{pmatrix}
    a_{1}&\ldots&a_{d+2}\\
    1 &\ldots& 1\\
    \omega_{a_{1}}&\ldots&\omega_{a_{d+2}}
  \end{pmatrix}
  \ \leq \ 0 \enspace .
\end{equation}
So $\omega\in C(T)$ if and only if $\omega$ satisfies the
inequalities~\eqref{eq:GKZconeineq} for all pairs of $d$-simplices
sharing a $(d-1)$-face. More generally, one obtains the following
characterization for arbitrary subdivisions:

\begin{lemma}
  \label{lem:GKZcones}
  Let $S$ be a polytopal subdivision of $(Q,A)$ and let $T$ be a
  triangulation of $(Q,A)$ refining $S$. Then the following statements
  for $\omega\in\R^{A}$ are equivalent:
  \begin{compactenum}[(i)]
  \item $\omega\in C(S)$
  \item For any two $d$-simplices
    $ [a_{1},\ldots,a_{d+1}],[a_{2},\ldots,a_{d+2}]\in T $ sharing a
    $(d-1)$-face, where $[a_{1},\ldots,a_{d+1}]$ is positively
    oriented, $\omega$ satisfies inequality~\eqref{eq:GKZconeineq},
    and equality holds if both $d$-simplices of $T$ are contained in the
    same $d$-cell of~$S$.
  \end{compactenum}
\end{lemma}

In particular, Lemma~\ref{lem:GKZcones} implies that the secondary
cones are closed polyhedral cones. The following lemma is also not
difficult to see:

\begin{lemma}
  \label{lem:GKZtopdim}
  A secondary cone $C(S)$ has non-empty interior in $\R^{A}$ if and
  only if the polytopal subdivision $S$ of $(Q,A)$ is in fact a
  regular triangulation with vertex set equal to~$A$.
\end{lemma}

We finally arrive at the first fundamental result of the classical
GKZ-theory:
\begin{thmdef}[secondary fan]
  \label{thm:secfan}
  The collection of secondary cones of regular subdivisions,
  \begin{equation}
    \label{eq:GKZsecfan}
    \secfan(A)=
    \big\{
    ~C(D)~|~
    D\text{ is a regular subdivision of (Q,A) }
    \big\}
  \end{equation}
  is a polyhedral fan with support $\R^{A}$, called the
  \emph{secondary fan} of the point configuration
  $A\subseteq\R^{d}$. More specifically, the following holds for all
  $\omega\in\R^{A}$ and all regular subdivisions $D_{1}$ and $D_{2}$
  of $(Q,A)$:
  \begin{compactenum}
  \item Every $\omega\in\R^{A}$ is contained in
    $C(D(\omega))\in\secfan(A)$. 
  \item $D_{1}\preceq D_{2}$ if and only if $C(D_{2})$ is a face of
    $C(D_{1})$. In particular, $C(D_{1})=C(D_{2})$ implies $D_{1}=D_{2}$.
  \item There is a uniquely determined finest common coarsening
    $D_{3}$ of $D_{1}$ and $D_{2}$ among all regular subdivisions, and
    $C(D_{1})\cap C(D_{2})=C(D_{3})$.
  \end{compactenum}
  Moreover, 
  \begin{compactitem}[(iv)]
  \item the top-dimensional cones in the secondary fan $\secfan(A)$
    are precisely the secondary cones of regular triangulations with
    vertex set $A$.
  \end{compactitem}
\end{thmdef}

\begin{remark}
  The analogous statements to Lemma~\ref{lem:GKZtopdim} and hence
  Theorem~\ref{thm:secfan}~(iv) do not hold in the setting of
  punctured Riemann surfaces (cf.\ Section~\ref{sec:fan}). A
  top-dimensional cone in the secondary fan of a punctured Riemann
  surface may correspond to a Delaunay decomposition that is not an
  ideal triangulation. In other words, not all finest Delaunay
  decompositions are ideal triangulations.
\end{remark}

For every triangulation $T$ of $(Q,A)$ 
let the function $L_{T}:\R^{A}\rightarrow\R$ be defined by
\begin{equation}
  \label{eq:LTomega}
  \begin{split}
    L_{T}(\omega) \ &= \ \int_{Q}g_{T,\omega}(x)\,dx\\
    &= \ \frac{1}{d+1}
    \sum_{[a_{1},\ldots,a_{d+1}]\in T}
    \vol([a_{1},\ldots,a_{d+1}])(\omega_{a_{1}}+\ldots+\omega_{a_{d+1}}) \enspace,
  \end{split}
\end{equation}
where all simplices are positively oriented.
Since $L_{T}$ is a linear function for every triangulation $T$, we can
interpet $L$ as the function 
\begin{equation*}
    L:\Triang(A)\longrightarrow (\R^{A})^{*},\quad
    T\longmapsto L_{T} \enspace,
\end{equation*}
where $\Triang(A)$ is the set of triangulations of $(Q,A)$
and~$(\R^{A})^{*}$ is the dual vector space of $\R^{A}$.  The
coordinate vector of $(d+1)L_{T}$ with respect to the canonical basis
of $(\R^{A})^{*}$,
\begin{equation}
  \label{eq:ella}
  (\ell_{a}(T))_{a\in A}\in\R^{A},
  \quad
  \ell_{a}(T)\ =\sum_{\sigma\in T:a\in\sigma^{}}
  \vol(\sigma)\enspace,
\end{equation}
is  called the \emph{GKZ-vector} of the triangulation
$T$ and often identified with a vector in $\R^{n}$ via a numbering of the
elements of $A$.

\begin{definition}(secondary polytope)
  The \emph{secondary polytope} $\secpoly(A)$ is the convex hull of
  the linear functionals $L_{T}$, i.e.,
  \begin{equation}
    \label{eq:GKZsecpoly}
    \secpoly(A) \ = \ \conv\big(\{L_{T}\}_{T\in\Triang(A)}\big)
    \subseteq(\R^{A})^{*} \enspace.
  \end{equation}
\end{definition}

For a (bounded or unbounded) polytope $P\subseteq V$ in a finite
dimensional real vector space $V$, and a face $F$ of $P$, the
\emph{normal cone} $N_{P}(F)$ is the cone of all functionals in the
dual vector space $V^{*}$ that attain their maximal value in $P$ at
all points of $F$. The dimensions are complementary:
\begin{equation*}
  \dim F+\dim C_{N}(F) \ = \ \dim V \enspace.
\end{equation*}
The \emph{normal fan} of $P$ is the fan containing the normal cones of
all faces $P$. Note that the normal fan is complete, i.e.,
$\bigcup_{F}N_{P}(F)=V$, if and only if the polytope $P$ is bounded.

The following theorem is the second fundamental result of the
classical GKZ-theory.
\begin{theorem}
  The normal fan of the secondary polytope $\secpoly(A)$ is the
  secondary fan $\secfan(A)$.
\end{theorem}
This follows directly from the definitions and the fact that a
triangulation $T$ maximizes the integral in~\eqref{eq:LTomega} if and
only if $g_{T,\omega}$ is concave.

It should be mentioned that the dimension of the secondary polytope of
$A$ is strictly less than $n=\dim(\R^{A})^{*}$, corresponding to the
fact that all cones of the secondary fan contain a nontrivial subspace
of $\R^{A}$. Indeed, if $\omega\in\R^{A}$ is the restriction $h|_{A}$
of an affine function~\eqref{eq:haff} to $A$, then
$g_{\omega}=h|_{Q}$, and therefore $D(\omega)$ is the face lattice of
$Q$ itself, and $\omega\in C(S)$ for every subdivision $S$
of $(Q,A)$. Thus, all secondary cones contain the $(d+1)$-dimensional
linear subspace 
\begin{equation}
  \label{eq:lineality}
  \mathcal{L}_{A}\ =\ \big\{~\omega\in\R^{A}~\big|~
  \omega\ =\ h|_{A}~\text{for some affine function}~h:\R^{d}\rightarrow\R
  \big\}\subseteq\R^{A}\enspace.
\end{equation}
Also, if $\omega=h|_{A}$ then for any triangulation $T$,
$L_{T}(\omega)=h(b)$, where $b$ is the barycenter of $Q$. Therefore
the secondary polytope is contained in an affine subspace
of~$(\R^{A})^{*}$ that is spanned by the $(n-d-1)$-dimensional
annihilator~$\mathcal{L}_{A}^{\perp}$.
In contrast, our secondary polyhedra of punctured Riemann surfaces will turn out to be full-dimensional.

\section{The Hyperbolic Plane}\label{sec:hyperbolic}
\noindent We continue with a brief discussion of hyperbolic geometry in order to introduce our notation and terminology.
For further reading we suggest \cite{Flavors}, \cite{Thurston} or \cite{Katok}.
\emph{Minkowski $3$-space}, denoted as $\R^{2,1}$, is the
three-dimensional real vector space together with the indefinite
scalar product
\[
  \scalp{~(x_1,x_2,x_3)~,~(x_1^\prime,x_2^\prime,x_3^\prime)~} \ = \  x_1x_1^\prime + x_2x_2^\prime - x_3x_3^\prime \enspace.
\]
The set of points in $\R^{2,1}$ with $\scalp{x,x}=-1$ is the standard hyperboloid of two sheets.
Its upper sheet
\[ \HH \ = \ \{~ x ~\in~ \R^{2,1} ~|~ \scalp{x,x}=-1,~x_3>0 ~\} \]
is equipped with the Riemannian metric induced by restricting the
Minkowski scalar product to tangent hyperplanes.
This gives rise to the \emph{hyperboloid model} of the hyperbolic plane.
Its points are the points in $\HH$, and geodesics (or \emph{[hyperbolic]
  lines}) are the intersections of $\HH$ with planes through the origin.
The group $\text{SO}^+(2,1)$ of linear transformations with
determinant $1$ that preserve the Minkowski scalar product and map the
upper sheet of the hyperboloid to itself acts on the hyperbolic plane
$\HH$ as the group of orientation preserving isometries.

A ray in the \emph{positive light cone}  
\[ \Lplus \ = \ \{~ v ~\in ~ \R^{2,1} ~|~ \scalp{v,v}=0,~ v_3>0 ~\} \]
is called an \emph{ideal point} of $\HH$.
The set of ideal points $\mathbb{S}^1_\infty=\Lplus/\R_{>0}$ is called the \emph{ideal boundary} of $\HH$.
Two distinct ideal points span a hyperplane through the origin.
The intersection of that hyperplane with $\HH$ yields the geodesic connecting the two ideal points.
Furthermore, each point $v \in \Lplus$ defines the \emph{horocycle} 
\[ h(v) \ = \  \{~ x \in \HH ~|~ \scalp{x,v}=-1 ~\} \enspace, \]
centered at the ideal point $\R_{>0}v$.  
Horocycles are limiting cases of circles in the hyperbolic plane as their radii tend to infinity and their centers tend to an ideal point.

Other models of the hyperbolic plane arise via projections. 
These include the ones below, all of which are rotationally symmetric with respect to the vertical axis; see Figure~\ref{fig:hyperbolic_models} for a sketch.
In each case the hyperbolic metric is carried over by the respective projection.
\begin{enumerate}[label = (\roman*)]
\item The \emph{Beltrami--Klein model} (or \emph{projective model})
  $\Klein$ is obtained by projecting $\HH$ from the origin onto the open unit disk at height $1$. 
The ideal boundary is projected to the topological boundary of this disk. 
Geodesics are Euclidean line segments within $\Klein$.	

\item The \emph{hemisphere model} $\hemi$ is obtained from $\HH$ by stereographically projecting through the point $(0,0,-1)$ onto the northern hemisphere of the unit sphere. 
Its equator forms the ideal boundary and geodesics are intersections of the hemisphere with hyperplanes orthogonal to the $(x_1,x_2)$-plane.

\noindent%
Projection along vertical lines maps directly from the Beltrami--Klein model to
the hemisphere model and vice versa. That is, a point $(x_{1},x_{2})$
in the Beltrami--Klein model corresponds to the point
$(x_{1},x_{2},\sqrt{1-x_{1}^{2}-x_{2}^{2}})$ in the hemisphere model.

\item The \emph{half-plane model} $\HS$ is obtained from $\hemi$ by stereographic projection through the point $(0,-1,0)$ onto the upper half-plane $\HS = \{ x \in \R^3 \,|\, x_2=0,\,x_3>0 \}$, which is identified with the upper half-plane of the complex plane $\mathbb{C}$. 
In this model, the ideal boundary is $\R \cup \{\infty\}$, geodesics
are half circles orthogonal to the real axis or Euclidean vertical
lines, and horocycles appear as Euclidean circles tangent to the real line $\R$ or as horizontal lines.
The group of orientation preserving isometries becomes the group $\PSL2R=\text{SL}_2(\R)/{\pm \text{Id}}$ of fractional linear transformations.
\end{enumerate}

While we mostly work with the hyperboloid and half-plane model, the
hemisphere model and its \enquote{light cylinder}~\eqref{eq:Lplushemi}
play an important role in the construction of secondary polyhedra in Section~\ref{sec:poly}.
The stereographic projection mapping the hyperboloid $\HH$ to the
hemisphere $\hemi$ (cf.\ Figure~\ref{fig:hyperbolic_models}) is the restriction of a projective
transformation to $\R^{3}$.
In affine coordinates it is given by
\begin{equation}\label{eq:trafo}
 (x_1,x_2,x_3) ~ \longmapsto ~ \frac{1}{x_3}(x_1,x_2,1)\enspace.
\end{equation}
From the standpoint of projective geometry, the hyperboloid model
$\HH$ and the hemisphere model $N$ are therefore just different affine
views of the same projective model. The
transformation~\eqref{eq:trafo} maps the positive light cone $\Lplus$
to the cylinder
\begin{equation}
  \label{eq:Lplushemi}
  \Lplus_\hemi \ = \ \{~x\in\R^{3}~|~x_{1}^{2}+x_{2}^{2}=1,~x_{3}>0~\}\enspace.
\end{equation}
Points in $\Lplus_\hemi$ represent horocycles in the hemisphere
model. Explicitly, $w \in \Lplus_\hemi$ corresponds to the 
horocycle
\begin{equation*}
h_\hemi(w) \ = \  \{~ x \in \hemi ~|~ x_1w_1+x_2w_2+x_3w_3=1 ~\}\enspace.
\end{equation*}

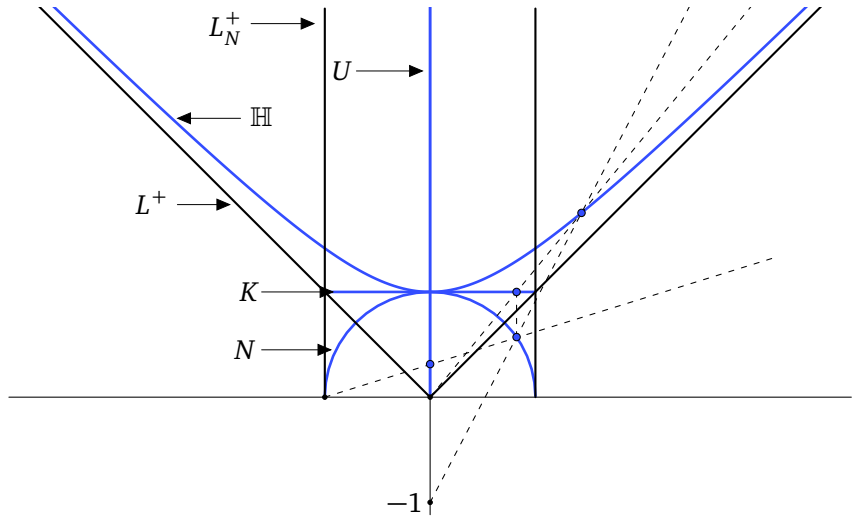
\begin{figure}[bth]
  \begin{tikzpicture}[scale=1.4,line cap=round,line join=round,>=triangle 45,x=1.0cm,y=1.0cm]
\clip(-4,-1.2) rectangle (4,3.7);
\draw (-4,0) -- (4,0);
\draw (0.,-1.1171008220908596) -- (0.,3.693037482659128);
\draw[line width=1pt,color=blue2,smooth,samples=100,domain=-4.0790145977108345:4.142740915136532] plot(\x,{sqrt((\x)^(2.0)+1.0)});
\draw [shift={(0.,0.)},line width=1pt,color=blue2]  plot[domain=0.01:3.131592653589793,variable=\t]({1.*1.*cos(\t r)+0.*1.*sin(\t r)},{0.*1.*cos(\t r)+1.*1.*sin(\t r)});
\draw [line width=1pt,color=blue2] (-1.,1.)-- (1.,1.);
\draw [line width=1.1pt,color=blue2] (0,0.01)-- (0,5);
\draw [line width=0.8pt,domain=0.0:4.142740915136532] plot(\x,{(-0.--3.*\x)/3.});
\draw [line width=0.8pt,domain=-4.0790145977108345:0.0] plot(\x,{(-0.--3.*\x)/-3.});
\draw [line width=0.8pt] (-1.,0.) -- (-1.,3.693037482659128);
\draw [line width=0.8pt] (1.,0.) -- (1.,3.693037482659128);
\draw [->,line width=0.1pt] (-1.8,2.65) -- (-2.4,2.65);
\draw [->,line width=0.1pt] (-1.6,1) -- (-0.92,1);
\draw [->,line width=0.1pt] (-2.4,1.83) -- (-1.9,1.83);
\draw [->,line width=0.1pt] (-1.6,0.45) -- (-0.94,0.45);
\draw [->,line width=0.1pt] (-1.7,3.55) -- (-1.1,3.55);
\draw [->,line width=0.1pt] (-0.7,3.1) -- (-.07,3.1);
\draw [dash pattern=on 2pt off 3pt] (-1,0) -- (4*0.8209228226113612,1.33);
\draw [dash pattern=on 2pt off 3pt,domain=0.0:4.142740915136532] plot(\x,{(-1.437594623649301--2.7511933936448525*\x)/1.437594623649301});
\draw [dash pattern=on 2pt off 3pt,domain=0.0:4.142740915136532] plot(\x,{(-0.--1.7511933936448525*\x)/1.437594623649301});
\draw [dash pattern=on 2pt off 3pt] (0.8209228226113612,1.)-- (0.8209228226113612,0.5710391574277508);
\draw (-2.2,3.75) node[anchor=north west] {$\Lplus_\hemi$};
\draw (-1.8,2.66) node[anchor=west] {$\HH$};
\draw (-2.9,2.05) node[anchor=north west] {$\Lplus$};
\draw (-1.9,1) node[anchor=west] {$K$};
\draw (-1.75,0.63) node[anchor=north] {$\hemi$};
\draw (-0.5,-1) node[anchor=west] {$-1$};
\draw (-1.03,3.1) node[anchor=west] {$\HS$};
\draw [fill=blue2] (1.437594623649301,1.7511933936448525) circle (1.0pt);
\draw [fill=black] (0.,-1.) circle (0.6pt);
\draw [fill=black] (0,0) circle (0.6pt);
\draw [fill=black] (-1,0) circle (0.6pt);
\draw [fill=blue2] (0,0.314) circle (1.0pt);
\draw [fill=blue2] (0.8209228226113612,1.) circle (1.0pt);
\draw [fill=blue2] (0.8209228226113612,0.5710391574277508) circle (1.0pt);
\end{tikzpicture}
  \caption{Projections between the hyperboloid ($\HH$), Beltrami--Klein ($\Klein$), hemisphere ($\hemi$) and half-plane ($\HS$) models of the hyperbolic plane.}
  \label{fig:hyperbolic_models}
\end{figure}

Consider two horocycles $h_1, h_2$ centered at two distinct ideal points.
The \emph{signed hyperbolic distance} $\ell$ between $h_1$ and $h_2$ is measured along the geodesic connecting their centers, and the sign is taken negative if and only if the horocycles intersect, see Figure~\ref{fig:signed_distance}.
Following Penner \cite[Chapter~1, \S4.1]{pennerbook}, we define the \emph{$\lambda$-length} of $h_1, h_2$ to be
\[ \lambda(h_1,h_2) \ = \ e^{\ell/ 2} \enspace .\] 
If the two horocycles are given as $h_1=h(v_1)$ and $h_2=h(v_2)$ for
two light cone vectors $v_1,v_2 \in \Lplus$, the former definition yields
\begin{equation}\label{eq:lambda_lightcone}
	\lambda(h_1,h_2) \ = \ \sqrt{- \tfrac{1}{2}\scalp{v_1,v_2}} \enspace .
\end{equation}

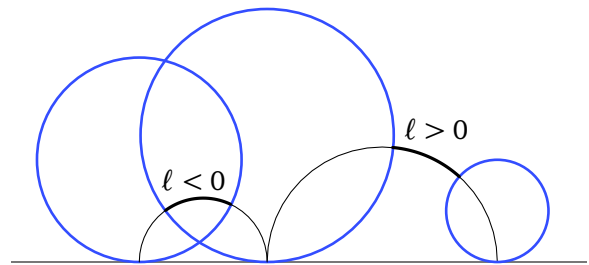
\begin{figure}[bth]
	\centering
\begin{tikzpicture}[scale=1.7,line cap=round,line join=round,>=triangle 45,x=1.0cm,y=1.0cm]
\clip(-2,-0.3) rectangle (2.5,2.2);
\draw (-4,0.)-- (6.2,0.);
\draw [shift={(0.9,0.)}] plot[domain=0.:3.141592653589793,variable=\t]({1.*0.9*cos(\t r)+0.*0.9*sin(\t r)},{0.*0.9*cos(\t r)+1.*0.9*sin(\t r)});
\draw [color=blue2,line width=1pt] (0.,0.99) circle (0.99cm);
\draw [color=blue2,line width=1pt] (1.8,0.4) circle (0.4cm);
\draw [color=blue2,line width=1pt] (-1,0.8) circle (0.8cm);
\draw [shift={(0.9,0.)},line width=1.2pt]  plot[domain=0.8364486591584586:1.48,variable=\t]({1.*0.9*cos(\t r)+0.*0.9*sin(\t r)},{0.*0.9*cos(\t r)+1.*0.9*sin(\t r)});
\draw [shift={(-0.5,0.)}] plot[domain=0.:3.141592653589793,variable=\t]({1.*0.5*cos(\t r)+0.*0.5*sin(\t r)},{0.*0.5*cos(\t r)+1.*0.5*sin(\t r)});
\draw [shift={(-0.5,0.)},line width=1.2pt] plot[domain=1.13:2.2,variable=\t]({1.*0.5*cos(\t r)+0.*0.5*sin(\t r)},{0.*0.5*cos(\t r)+1.*0.5*sin(\t r)});
\draw (1,1.2) node[anchor=north west] {$\ell>0$};
\draw (-.9,.8) node[anchor=north west] {$\ell<0$};

\end{tikzpicture}
	\caption{Signed distances of horocycles and $\lambda$-lengths.}
	\label{fig:signed_distance}
\end{figure}
 
An \emph{ideal triangle} $\itr$ is the closed region in the hyperbolic plane that is bounded
by three geodesics $a,b,c$ (the \emph{sides}) connecting three ideal points (the \emph{vertices}).
A \emph{decoration} of an ideal triangle is a triple of horocycles centered at the vertices.
Such a decoration gives rise to three $\lambda$-lengths $\lambda(a), \lambda(b)$ and $\lambda(c)$, one along each edge, see Figure~\ref{fig:dec_tri}.
The hyperbolic length of the horocyclic arc within the triangle at a
vertex is called the \emph{$h$-length} at the vertex.
In the ``trigonometry'' of decorated ideal triangles, $\lambda$-lengths and
$h$-lengths play a role similar to the side lengths and angles of
ordinary trigonometry.

Later we will consider ideal triangulations of punctured Riemann
surfaces, and then some of the vertices of a triangle may correspond to the
same ideal point of the surface.
For this reason, we will need more sophisticated notation than just
labelling the $h$-lengths by an incident triangle-vertex pair.
The standard orientation of the hyperbolic plane induces a cyclic order of the
sides of a triangle.
If the sides are $a,b,c$ in this cyclic order,
we label the three $h$-lengths of $\itr$ by $\hl^\itr_{ab},
\hl^\itr_{bc}$ and $\hl^\itr_{ca}$, see Figure~\ref{fig:dec_tri}.
The $\lambda$- and the $h$-lengths of a decorated ideal triangle are related via
\begin{equation}\label{eq:h-lengths}
\hl^\itr_{ab} \ = \ \frac{\lambda_{c}}{\lambda_{a} \lambda_{b}} \enspace ,
\end{equation}
see \cite[Chapter~1, \S4.2]{pennerbook}.

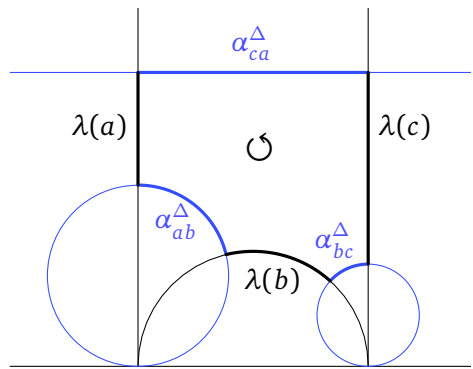
\begin{figure}[bth]
  \begin{tikzpicture}[scale=1.7,line cap=round,line join=round,>=triangle 45,x=1.0cm,y=1.0cm]
\clip(-1.3,-0.3) rectangle (2.6,2.8);
\draw (-1,0.)-- (6.2,0.);
\draw (0.,0.)-- (0.,3.6);
\draw (1.8,0) -- (1.8,4);
\draw [shift={(0.9,0.)}] plot[domain=0.:3.141592653589793,variable=\t]({1.*0.9*cos(\t r)+0.*0.9*sin(\t r)},{0.*0.9*cos(\t r)+1.*0.9*sin(\t r)});
\draw [color=blue2] (0.,0.7085636479715056) circle (0.7085636479715056cm);
\draw [color=blue2] (1.8,0.4) circle (0.4cm);
\draw [color=blue2] (-1,2.3)-- (6.2,2.3);
\draw [line width=1.2pt,color=blue2] (0,2.3)-- (1.8,2.3);
\draw [shift={(1.8,0.4)},line width=1.2pt,color=blue2]  plot[domain=1.6:2.38,variable=\t]({1.*0.4*cos(\t r)+0.*0.4*sin(\t r)},{0.*0.4*cos(\t r)+1.*0.4*sin(\t r)});
\draw [shift={(0.,0.71)},line width=1.2pt,color=blue2]  plot[domain=0.24:1.55,variable=\t]({1.*0.7085636479715057*cos(\t r)+0.*0.7085636479715057*sin(\t r)},{0.*0.7085636479715057*cos(\t r)+1.*0.7085636479715057*sin(\t r)});
\draw [line width=1.2pt] (0.,2.3)-- (0.,1.4171272959430112);
\draw [line width=1.2pt] (1.8,2.3)-- (1.8,0.8);
\draw [shift={(0.9,0.)},line width=1.2pt]  plot[domain=0.8364486591584586:1.8077035198040758,variable=\t]({1.*0.9*cos(\t r)+0.*0.9*sin(\t r)},{0.*0.9*cos(\t r)+1.*0.9*sin(\t r)});

\draw (0.75,0.85) node[anchor=north west] {$\lambda(b)$};
\draw (-0.6,2.05) node[anchor=north west] {$\lambda(a)$};
\draw (1.8,2.05) node[anchor=north west] {$\lambda(c)$};
\node at (0.95,1.7) {\Large $\circlearrowleft$};
\draw [color=blue2text](0.05,1.35) node[anchor=north west] {$\hl^\itr_{ab}$};
\draw [color=blue2text](1.3,1.2) node[anchor=north west] {$\hl^\itr_{bc}$};
\draw [color=blue2text](0.65,2.74) node[anchor=north west] {$\hl^\itr_{ca}$};

\end{tikzpicture}
  \caption{Decorated (oriented) ideal triangle $\Delta=(a,b,c)$ in the upper half plane with $\lambda$- and $h$-lengths.}
  \label{fig:dec_tri}
\end{figure}

Ideal triangles generalize to arbitrary ideal polygons, with or without decorations.
An ideal quadrilateral admits exactly two triangulations.
If the edges $a,b,c,d$ are cyclically ordered, then each triangulation is determined by the choice of either the diagonal $e$, yielding the two triangles $\itr=(a,b,e)$ and $\itr'=(c,d,e)$, or the diagonal $f$, yielding the triangles $(a,f,d)$ and $(b,c,f)$, see Figure~\ref{fig:lambda_h_lengths}.
Substituting one diagonal by the other is referred to as a \emph{diagonal flip}.
A decoration with horocycles at the four vertices gives rise to six $\lambda$-lengths.
They satisfy the \emph{Ptolomy relation}
\begin{equation}\label{eq:ptolomy}
  \lambda(e)\lambda(f) \ = \ \lambda(a)\lambda(c) +
  \lambda(b)\lambda(d)\enspace,
\end{equation}
which follows immediately from~\eqref{eq:h-lengths}, see
Figure~\ref{fig:lambda_h_lengths} and \cite[Chapter~1, \S4.3]{pennerbook}.
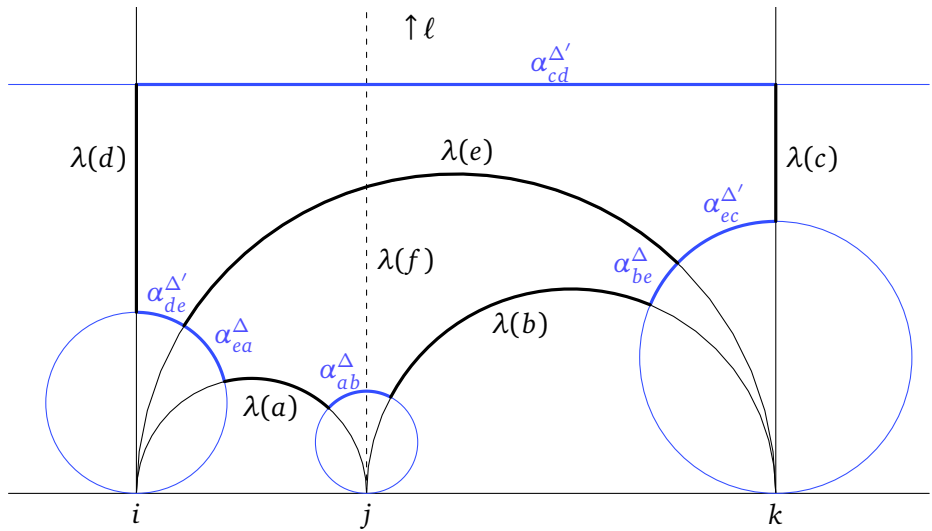
\begin{figure}[tbh]
  \begin{tikzpicture}[scale=1.7,line cap=round,line join=round,>=triangle 45,x=1.0cm,y=1.0cm]
\clip(-1.15,-0.3) rectangle (6.7,3.8);
\draw (-1,0.)-- (6.2,0.);
\draw (0.,0.)-- (0.,3.9);
\draw (5.,0.)-- (5.,3.9);
\draw[dash pattern = on 2pt off 3pt] (1.8,0) -- (1.8,4);

\draw [shift={(0.9,0.)}] plot[domain=0.:3.141592653589793,variable=\t]({1.*0.9*cos(\t r)+0.*0.9*sin(\t r)},{0.*0.9*cos(\t r)+1.*0.9*sin(\t r)});
\draw [shift={(3.4,0.)}] plot[domain=0.:3.141592653589793,variable=\t]({1.*1.6*cos(\t r)+0.*1.6*sin(\t r)},{0.*1.6*cos(\t r)+1.*1.6*sin(\t r)});
\draw [shift={(2.5,0.)}] plot[domain=0.:3.141592653589793,variable=\t]({1.*2.5*cos(\t r)+0.*2.5*sin(\t r)},{0.*2.5*cos(\t r)+1.*2.5*sin(\t r)});
\draw [color=blue2] (0.,0.7085636479715056) circle (0.7085636479715056cm);
\draw [color=blue2] (1.8,0.4) circle (0.4cm);
\draw [color=blue2] (5.,1.064438472742977) circle (1.064438472742977cm);
\draw [color=blue2] (-1,3.2)-- (6.2,3.2);

\draw [line width=1.2pt,color=blue2] (0.,3.2)-- (5.,3.2);
\draw [shift={(1.8,0.4)},line width=1.2pt,color=blue2]  plot[domain=1.0808390005411663:2.4072449859533545,variable=\t]({1.*0.4*cos(\t r)+0.*0.4*sin(\t r)},{0.*0.4*cos(\t r)+1.*0.4*sin(\t r)});
\draw [shift={(0.,0.7085636479715056)},line width=1.2pt,color=blue2]  plot[domain=0.2369071930091794:1.55,variable=\t]({1.*0.7085636479715057*cos(\t r)+0.*0.7085636479715057*sin(\t r)},{0.*0.7085636479715057*cos(\t r)+1.*0.7085636479715057*sin(\t r)});
\draw [shift={(5.,1.064438472742977)},line width=1.2pt,color=blue2]  plot[domain=1.5707963267948966:2.744872049573939,variable=\t]({1.*1.064438472742977*cos(\t r)+0.*1.064438472742977*sin(\t r)},{0.*1.064438472742977*cos(\t r)+1.*1.064438472742977*sin(\t r)});

\draw [line width=1.2pt] (0.,3.2)-- (0.,1.4171272959430112);
\draw [line width=1.2pt] (5.,3.2)-- (5.,2.128876945485954);
\draw [shift={(0.9,0.)},line width=1.2pt]  plot[domain=0.8364486591584586:1.8077035198040758,variable=\t]({1.*0.9*cos(\t r)+0.*0.9*sin(\t r)},{0.*0.9*cos(\t r)+1.*0.9*sin(\t r)});
\draw [shift={(3.4,0.)},line width=1.2pt]  plot[domain=1.1740757227790417:2.651635327336065,variable=\t]({1.*1.6*cos(\t r)+0.*1.6*sin(\t r)},{0.*1.6*cos(\t r)+1.*1.6*sin(\t r)});
\draw [shift={(2.5,0.)},line width=1.2pt]  plot[domain=0.8050544512459804:2.589228059352916,variable=\t]({1.*2.5*cos(\t r)+0.*2.5*sin(\t r)},{0.*2.5*cos(\t r)+1.*2.5*sin(\t r)});

\draw (0.75,0.85) node[anchor=north west] {$\lambda(a)$};
\draw (2.7,1.5) node[anchor=north west] {$\lambda(b)$};
\draw (5,2.8) node[anchor=north west] {$\lambda(c)$};
\draw (-0.6,2.8) node[anchor=north west] {$\lambda(d)$};
\draw [color=blue2text](3.65,2) node[anchor=north west] {$\hl^\itr_{be}$};
\draw [color=blue2text](0.53,1.45) node[anchor=north west] {$\hl^\itr_{ea}$};
\draw [color=blue2text](1.35,1.17) node[anchor=north west] {$\hl^\itr_{ab}$};
\draw [color=blue2text](4.35,2.5) node[anchor=north west] {$\hl^{\itr'}_{ec}$};
\draw [color=blue2text](-0.0,1.8) node[anchor=north west] {$\hl^{\itr'}_{de}$};
\draw [color=blue2text](3,3.6) node[anchor=north west] {$\hl^{\itr'}_{cd}$};
\draw (2.3,2.85) node[anchor=north west] {$\lambda(e)$};
\draw (1.8,2) node[anchor=north west] {$\lambda(f)$};
\draw (0,0) node[anchor=north] {$i$};
\draw (1.8,0) node[anchor=north] {$j$};
\draw (5,0) node[anchor=north] {$k$};
\node at (2.15,3.65) {$\uparrow$};
\draw (2.3,3.8) node[anchor=north] {$\ell$};

\end{tikzpicture}
  \caption{Decorated ideal quadrilateral in the upper half-plane with $\lambda$- and $h$-lengths.
  The diagonal $e$ induces a triangulation with triangles $\Delta=(a,b,e)$ and $\Delta'=(c,d,e)$.}
  \label{fig:lambda_h_lengths}
\end{figure}

\section{Punctured Riemann Surfaces}\label{sec:riemann}
\noindent
A \emph{Riemann surface} is a complex one-dimensional
manifold. Here we are only interested in
\emph{punctured Riemann surfaces}, i.e., compact Riemann surfaces with a
finite number $n\geq 1$ of points removed. Morevover, we consider only
punctured Riemann surfaces with negative Euler characteristic, which
excludes only spheres with one or two punctures. Under this
assumption, a punctured Riemann surface admits a conformal complete hyperbolic
metric with finite area, which is unique up isotopy. The punctures
of the Riemann surface correspond to \emph{cusps} of the hyperbolic
surface~(cf.\ Figure~\ref{fig:cusp}). Henceforth, we take the metric
point of view and consider punctured Riemann surfaces as complete
hyperbolic surfaces with finite area and at least one cusp.

Riemann surfaces arise in many concrete forms, and this accounts for
the enormous richness of the theory. In this section we briefly review
two of these forms:
quotients of the hyperbolic plane by groups of isometries, and surfaces constructed by gluing decorated ideal
triangles together. These two points of view are
particularly useful for the definition of the
secondary fan (cf.\ Section~\ref{sec:fan}) and the construction of
secondary polyhedra (cf.\ Section~\ref{sec:poly}).
For a more detailed discussion see, e.g., \cite{Beardon,Kapovich,Katok,Lehner,penner1,pennerbook}.
Instead of the half-plane model with isometry group $\PSL2R$, which is
the more classical approach, we use the hyperboloid model $\HH$ with isometry
group $\SO$ because this is how the Epstein--Penner convex hull construction is
usually described (cf.\ Section~\ref{sec:fan}). 

The group $\SO$ of orientation-preserving isometries of the hyperbolic
plane~$\HH$ acts naturally on the ideal boundary $S^{1}_{\infty}$.
An isometry is \emph{elliptic, parabolic or hyperbolic} if the number
of fixed ideal points is $0,1$ or $2$, respectively.
Only the identity fixes more than two ideal points. An elliptic
isometry has exactly one fixed point in $\HH$.
A \emph{Fuchsian group} is a discrete subgroup $\Gamma$ of $\SO$. 
It is a fundamental result that a subgroup of $\SO$ is discrete if and only if it acts properly discontinuously on $\HH$.
Furthermore, $\Gamma$ acts freely if and only if it does not contain any elliptic elements.
In this case, the quotient
\begin{equation}\label{eq:riemann}
  \Riem \ = \ \HH / \Gamma \enspace .
\end{equation}
with the induced hyperbolic metric is the Riemann surface defined by $\Gamma$.
Since $\HH$ is simply connected, $\Gamma$ is canonically isomorphic to
the fundamental group of $\Riem$ and acts by deck transformations. The
quotient $\HH/\Gamma'$ with repspect to another Fuchsian group $\Gamma'$ is
isometric to $\Riem$ if and only if $\Gamma$ and $\Gamma'$ are
conjugate subgroups of $\SO$.

In general, the area of the quotient $\Riem$ is not finite.
If it is finite, then the group $\Gamma$ is called a \emph{Fuchsian group of the first kind}.
In particular, this entails that $\Gamma$ is finitely
generated. Suppose further that $\Gamma$ contains a parabolic element
$\gamma$ with $p$ as its unique ideal fixed point.
Then the $\Gamma$-orbit of $p$ correponds to a \emph{cusp} of $\Riem$, see Figure~\ref{fig:cusp}.
From now on we suppose that the surface $\Riem = \HH / \Gamma$ has finite area and at least one cusp.
\begin{figure}[tbh]
  \mbox{
    \definecolor{ffffff}{rgb}{1.,1.,1.}
\begin{tikzpicture}[scale=1.5,line cap=round,line join=round,>=triangle 45,x=1.0cm,y=1.0cm]
\clip(-1.1,-0.6) rectangle (3.3,2.5);
\draw[color=white,fill=blue2!20] (-1.,3.5) -- (1.,3.5) -- (1.,0.) -- (-1.,0.) -- cycle;

\draw [dash pattern=on 2pt off 3pt] (0.,0.) -- (0.,2.7478120333134615);

\draw [dash pattern=on 2pt off 3pt] (2.,0.) -- (2.,2.7478120333134615);

\draw [shift={(-0.5,0.)},color=ffffff,fill=ffffff,fill opacity=1.0]  plot[domain=0.:3.141592653589793,variable=\t]({1.*0.5*cos(\t r)+0.*0.5*sin(\t r)},{0.*0.5*cos(\t r)+1.*0.5*sin(\t r)});
\draw [shift={(0.5,0.)},color=ffffff,fill=ffffff,fill opacity=1.0]  plot[domain=0.:3.141592653589793,variable=\t]({1.*0.5*cos(\t r)+0.*0.5*sin(\t r)},{0.*0.5*cos(\t r)+1.*0.5*sin(\t r)});

\draw [dd_middlearrow={latex},shift={(1.5,0.)}] plot[domain=0.:3.141592653589793,variable=\t]({-1.*0.5*cos(\t r)+0.*0.5*sin(\t r)},{0.*0.5*cos(\t r)+1.*0.5*sin(\t r)});
\draw [dd_middlearrow={latex},shift={(2.5,0.)}] plot[domain=0.:3.141592653589793,variable=\t]({1.*0.5*cos(\t r)+0.*0.5*sin(\t r)},{0.*0.5*cos(\t r)+1.*0.5*sin(\t r)});
\draw [dd_middlearrow={latex},shift={(-0.5,0.)}]  plot[domain=0.:3.141592653589793,variable=\t]({-1.*0.5*cos(\t r)+0.*0.5*sin(\t r)},{0.*0.5*cos(\t r)+1.*0.5*sin(\t r)});
\draw [dd_middlearrow={latex},shift={(0.5,0.)}]  plot[domain=0.:3.141592653589793,variable=\t]({1.*0.5*cos(\t r)+0.*0.5*sin(\t r)},{0.*0.5*cos(\t r)+1.*0.5*sin(\t r)});
\draw [shift={(0.165,0.)}] plot[domain=0.:3.141592653589793,variable=\t]({1.*0.165*cos(\t r)+0.*0.165*sin(\t r)},{0.*0.165*cos(\t r)+1.*0.165*sin(\t r)});
\draw [shift={(0.415,0.)}] plot[domain=0.:3.141592653589793,variable=\t]({1.*0.085*cos(\t r)+0.*0.085*sin(\t r)},{0.*0.085*cos(\t r)+1.*0.085*sin(\t r)});
\draw [shift={(0.75,0.)}] plot[domain=0.:3.141592653589793,variable=\t]({1.*0.25*cos(\t r)+0.*0.25*sin(\t r)},{0.*0.25*cos(\t r)+1.*0.25*sin(\t r)});
\draw [shift={(0.25,0.)},dash pattern=on 2pt off 3pt]  plot[domain=0.:3.141592653589793,variable=\t]({1.*0.25*cos(\t r)+0.*0.25*sin(\t r)},{0.*0.25*cos(\t r)+1.*0.25*sin(\t r)});
\draw (-2,0) -- (4,0);
\draw[middlearrow={latex}] (-1.,0.) -- (-1.,2.7478120333134615);
\draw [middlearrow={latex}] (1.,0.) -- (1.,2.7478120333134615);
\draw [middlearrow={latex}] (3.,0.) -- (3.,2.7478120333134615);
\begin{scriptsize}
\draw [fill=black] (-1.,0.) circle (0.5pt);
\draw [fill=black] (-1.,3.5) circle (0.5pt);
\draw [fill=black] (0.,0.) circle (0.5pt);
\draw [fill=black] (1.,0.) circle (0.5pt);
\draw [fill=black] (1.,3.5) circle (0.5pt);
\draw [fill=black] (2.,0.) circle (0.5pt);
\draw [fill=black] (3.,0.) circle (0.5pt);
\draw [fill=black] (0.5,0.) circle (0.5pt);
\draw [fill=black] (0.33,0.) circle (0.5pt);

\node at (0,-0.2) {$0$};
\node at (-1,-0.2) {$-1$};
\node at (1,-0.2) {$1$};
\node at (0.5,-0.2) {$\tfrac{1}{2}$};
\node at (0.3,-0.2) {$\tfrac{1}{3}$};
\node at (2,-0.2) {$2$};
\node at (3,-0.2) {$3$};
\end{scriptsize}
\node at (0.3,2) {Id};
\node at (2.3,2) {$A$};
\node at (0.5,0.3) {$B$};
\end{tikzpicture}
    \begin{tikzpicture}[scale = 0.9]
	\centering
	\path [out=251,in=321] (1,3) edge (-3,3);
	\path [out=245,in=127] (1.05,3) edge (0.9,-2);
	\path [out=320,in=124] (-3.05,3) edge (0.85,-2);
	\path [out=110,in=110] (-0.42,-0.05) edge (0.2,0.15);
	\path [out=-51,in=-95,dashed] (-0.42,-0.05) edge (0.22,0.17);
	\path [out=225,in=220,dashed] (0.22,2.19) edge (0.52,1.7);
	\path [out=10,in=70] (0.22,2.19) edge (0.51,1.7);
	\path [out=-45,in=-15,dashed] (-1.85,1.85) edge (-1.65,2.27);
	\path [out=120,in=170] (-1.85,1.85) edge (-1.65,2.27);

	\draw [->] (1.05,3.1) -- (1.125,3.25);
	\node [font = \scriptsize] at (1.3,3.4) {$p_1$};
	\draw [->] (-3.04,3.04) -- (-3.17,3.17);
	\node [font = \scriptsize] at (-3.3,3.3) {$p_2$};
	\draw [->] (0.92,-2.05) -- (1.02,-2.19);
	\node [font = \scriptsize] at (1.1,-2.36) {$p_3$};	
						
\end{tikzpicture}}
  \caption{Fuchsian group generated by the two parabolic transformations $A(z)=z+2$ and $B(z)=\tfrac{z}{2z+1}$ as a subgroup of $\PSL2R$ (left). The blue shaded redion is a fundamental domain. The resulting Riemann surface is a sphere with three cusps (right).}
  \label{fig:cusp}
\end{figure}
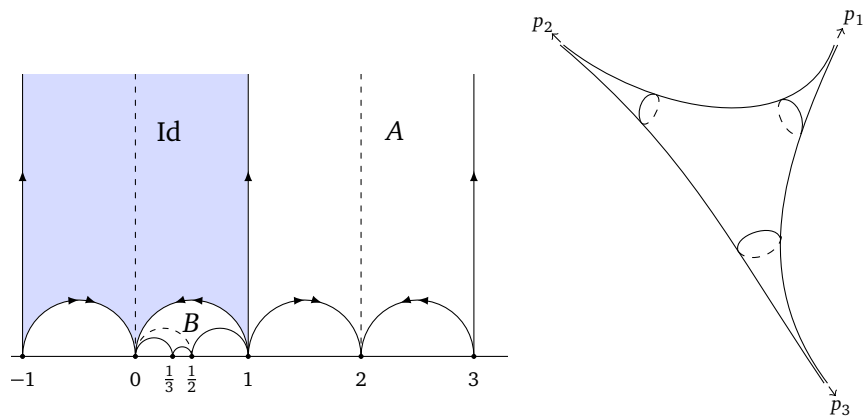

We will now discuss how the notions of horocycles, $\lambda$-lengths and $h$-lengths (cf.\
Section~\ref{sec:hyperbolic}) carry over to decorated Riemann surfaces
and ideal triangulations. This will lead to the second
point of view, surfaces glued from ideal triangles, and to some
practical formulas, which we will use in Sections~\ref{sec:fan}
and~\ref{sec:poly}.

A \emph{horocycle} at a cusp of $\Riem$ is a $\Gamma$-orbit of
horocycles centered at the parabolic fixed points of the
corresponding $\Gamma$-orbit.
If such a horocycle in $\HH$ is sufficiently small then its image under the projection to $\Riem$ is an embedded closed curve around the cusp.
A \emph{decoration} of $\Riem$ is a choice of one horocycle at each
cusp.

An \emph{ideal cell decomposition} of $\HH$ is a family of ideal polygons such that
\begin{inparaenum}
\item the polygons cover $\HH$ and
\item any two polygons intersect in a common edge, or the intersection is empty.
\end{inparaenum}
If each ideal polygon is an ideal triangle, the ideal cell decomposition is
called an \emph{ideal triangulation}.

An \emph{ideal cell decomposition} $S$ of $\Riem$ is a
$\Gamma$-invariant ideal cell decomposition of~$\HH$, all vertices of
which are parabolic fixed points of $\Gamma$. It decomposes $\Riem$
into finitely many ideal polygons, the \emph{faces} of $S$, which are glued
along their sides, the \emph{edges} of $S$. If all faces are
triangles, $S$ is an \emph{ideal triangulation} of $\Riem$.

Now let $\Riem$ be a punctured Riemann surface, decorated with a
horocycle at each cusp, and let $T$ be an ideal triangulation of
$\Riem$. The \emph{$\lambda$-length} $\lambda(e)$ of an edge $e$ of $T$ is
defined by~\eqref{eq:lambda_lightcone}, where $h_{1}$, $h_{2}$ are the
horocycles at the ends of a lift $\hat{e}$ of $e$ to the universal
cover $\HH$. 

Conversely, if $T$ is an ideal triangulation of the topological
surface $\Sgn$ of genus $g$ with $n$ punctures, a positive function
$\lambda:E_{T}\rightarrow\R_{>0}$ on the set of edges determines a
complete hyperbolic metric with finite area on $\Sgn$, uniquely up to
isotopy, together with decorating horocycles at the cusps.  Simply construct decorated ideal triangles with
$\lambda$-lengths determined by $\lambda$, one triangle for each face
of $T$, and glue them together according to the combinatorics of $T$
so that the horocycles fit toghether at the vertices.



The $h$-lengths of horocyclic arcs in the corners of the triangles of
$T$ are determined by the $\lambda$-lengths of the edges via
equation~\eqref{eq:h-lengths}. If $\Delta$ is a triangle of $T$ with
edges $a$, $b$, $c$ in the cyclic order induced by the orientation of
$\Riem$, we label the corners of $\Delta$ by $(\Delta, a, b)$,
$(\Delta, b, c)$, $(\Delta, c, a)$, and denote the corresponding
$h$-lengths by $\alpha^{\Delta}_{ab}$, $\alpha^{\Delta}_{bc}$,
$\alpha^{\Delta}_{ca}$.

\begin{remark}
  This notation is valid even if two sides of the triangle $\Delta$
  are glued together and correspond to the same edge of the
  triangulation $T$. For instance, if the edges of $\Delta$ are
  $a,a,c$ in cyclic order, then the three corners are labelled
  $(\Delta,a,a)$, $(\Delta,a,c)$ and $(\Delta,c,a)$.
\end{remark}

The total length of the decorating horocycle at the cusp $i$ is 
\begin{equation}\label{eq:omega_sum}
  \omega_i \ = \ \sum_{(\Delta,a,b)\sim i} \hl^\itr_{ab} \enspace ,
\end{equation}
where the sum is taken over all corners $(\itr,a,b)$ of $T$ incident
with cusp $i$. We call $\omega_{i}$ the \emph{weight} of cusp $i$ of
the decorated surface. Note that the weight of a cusp does not depend on the ideal triangulation. 

\section{The Secondary Fan}
\label{sec:fan}
\noindent
Epstein and Penner's convex hull construction~\cite{epstein_penner,penner1,pennerbook} is fundamental for the definition of the secondary fan of a punctured Riemann surface (cf.\ \ref{def:sfan}). This construction produces an ideal cell decomposition for each decorated Riemann surface $\Riem=\HH/\Gamma$ with $n \geq 1$ cusps in the following way.
By definition, the horocycle at the $i$th cusp corresponds to a $\Gamma$-orbit $\B_i$ of points in the positive light cone $\Lplus$.
Their union $\B = \bigcup_{i=1}^n \B_i$ is a countably infinite set as $\Gamma$ is finitely generated.
Now consider the Euclidean convex hull 
\begin{equation}
  \label{eq:Chull}
  \C \ = \ \conv(\B) \enspace.
\end{equation}
The following is a key observation.
\begin{proposition}[{\cite{epstein_penner}, \cite[Chapter 4, \S1]{pennerbook}}]\label{prop:convexhull_structure}
  The union $\B\subseteq\Lplus$ of $\Gamma$-orbits corresponding to a
  decoration of $\Riem=\HH/\Gamma$ is discrete and closed in $\Lplus
  \cup \{0\}$.  
  The faces of the boundary of the convex hull $\C$ project to an
  ideal cell decomposition of~$\Riem$.
\end{proposition}
Akiyoshi \cite{akiyoshi} generalized the convex hull construction for \emph{partially decorated} surfaces, i.e., decorations with horocycles at some, and at least one, of the cusps.
For each cusp $i$ that is not decorated, the induced decomposition of $\Riem$ has a punctured face containing the $i$th cusp in its interior.
Following \cite[\S\S 4--5]{springb_uniform} we call the ideal cell decomposition that arises from the convex hull construction the \emph{Delaunay decomposition} of the decorated Riemann surface.

We are interested in the different Delaunay decompositions of a fixed Riemann surface $\Riem$ with different decorations.  
To this end, we parametrize a decoration by its \emph{weight} vector
\begin{equation}
	\label{eq:weightvec}
	\omega \ = \ (\omega_1,\dots,\omega_n) ~\in~ \R_{\geq 0}^n \setminus \{0\} \enspace ,
\end{equation}
where $\omega_i$ is the length of the decorating horocycle at the $i$th cusp (cf.\ Section~\ref{sec:riemann}).
Zero weights correspond to undecorated cusps. The origin $\omega=0$ is not an admissible weight vector because the convex hull construction requires at least one decorated cusp.
\begin{definition}
  \label{def:cone}
  Let $\Riem$ be a Riemann surface with $n$ cusps.
  \begin{enumerate}
  \item For a weight vector~$\omega$ as in~\eqref{eq:weightvec}, let
    $\DD(\omega)$ be the Delaunay decomposition of $\Riem$ obtained
    for the corresponding decoration by the convex hull construction.
    
  \item For an ideal cell decomposition $S$ of $\Riem$, we call
    \[ \Cone(S) \ = \ \{~ \omega \in \Rpos^n\setminus \{0\} ~ | ~ S
    \preceq \DD(\omega) ~\} \]
    the \emph{secondary cone} of $S$, where we write $S_1 \preceq S_2$
    if $S_1$ refines $S_2$, i.e., every cell of $S_1$ is contained in
    some cell of $S_2$.
  \end{enumerate}	
\end{definition}

Note that the secondary cones do not contain $0$. This reflects the fact that the origin is not an admissible weight vector. 

We will see that the secondary cones are polyhedral cones in $\Rpos^n\setminus \{0\}$ (cf.\ Theorem~\ref{thm:cones}). This is a consequence of the following local characterization of Delaunay decompositions, see~\cite[Ch.~4, Lemma~1.7]{pennerbook} and~\cite{akiyoshi}.

\begin{definition}\label{def:local_Del}
	Let $T$ be an ideal triangulation of a decorated Riemann surface. 
	We say that an edge $e$ of $T$ \emph{satisfies the local Delaunay condition} if the sum of adjacent horocyclic arcs minus the sum of opposite horocyclic arcs is nonnegative, i.e., 
	\begin{equation}\label{eq:local_Del}
	\hl^\itr_{ea} + \hl^\itr_{be} + \hl^{\itr'}_{de} + \hl^{\itr'}_{ec} - \hl^\itr_{ab} - \hl^{\itr'}_{cd} \ \geq \ 0 \enspace ,
	\end{equation} 
	where $a$, $b$, $e$ and $c$, $d$, $e$ are the edges of the
        triangles $\Delta$ and $\Delta'$ containing $e$ (cf. Figure~\ref{fig:lambda_h_lengths}).
\end{definition}

\begin{lemma}\label{lem:local_Del}
	Let $S$ be an ideal cell decomposition of $\Riem$ and let $T$ be an ideal triangulation refining $S$. The the following statements for $\omega \in \Rpos^n\setminus \{0\}$ are equivalent:
	\begin{enumerate}
		\item $\omega\in\Cone(S)$
		\item Every edge $e$ of $T$ satisfies the local Delaunay condition~\eqref{eq:local_Del}, and equality holds if $e$ is not an edge of $S
		$. 
	\end{enumerate}	
\end{lemma}

\begin{theorem}\label{thm:cones}
  For any Delaunay decomposition $\DD$ of $\Riem$ the set $\Cone(\DD)$ is a closed polyhedral cone in $\Rpos^n \setminus \{0\}$. The faces of the secondary cone $\Cone(\DD)$ are precisely the secondary cones $\Cone(\DD')$ of Delaunay decompositions $\DD'$ that are refined by $\DD$.
\end{theorem}

\begin{proof}
  Let $T$ be an ideal triangulation refining $\DD$. 
  Denote the $h$-lengths at the corners of $T$ for the decoration with constant weight vector~$1$ by  $\bar\alpha$. Then the $h$-lengths for an arbitrary weight vector $\omega$ are $\alpha^{\Delta}_{ab}=\bar\alpha^{\Delta}_{ab}\,\omega_j$, where $j$ is the cusp at the corner $(\Delta, a, b)$.
  The local Delaunay condition~\eqref{eq:local_Del} becomes
  \begin{equation}
    \label{eq:edge_ineq}
    \left( \bar\hl^\itr_{ea} + \bar\hl^{\itr'}_{de} \right)\omega_i +  \left( \bar\hl^\itr_{be} + \bar\hl^{\itr'}_{ec} \right)\omega_k - \bar\hl^\itr_{ab}\omega_j - \bar\hl^{\itr'}_{cd}\omega_l 
    \ \geq \ 0
    \enspace ,
  \end{equation}
  where $i$, $j$, $k$, $l$ are the incident cusps as shown in Figure~\ref{fig:lambda_h_lengths}.  
  Thus, $\Cone(\DD)$ is the solution space of nonstrict homogeneous linear equalities and inequalities, hence a closed polyhedral cone. The second statement of the theorem is a consequence of the following observation, which follows directly from Definition~\ref{def:cone}:
   A Delaunay decomposition $\tilde D$ and a weight vector $\omega$ satisfy $\tilde D=\DD(\omega)$ if and only if $\tilde D$ is the coarsest Delaunay decomposition for which $\omega\in\Cone(\tilde D)$.
\end{proof}

\begin{corollary}
	Let $\DD_1, \DD_2$ be Delaunay decompositions. 
	\begin{enumerate}
		\item If $\Cone(D_1)\cap\Cone(D_2)\not=\emptyset$ then there is a unique finest common coarsening $D_3$ of $D_1$ and $D_2$ among all Delaunay decompositions, and  $\Cone(D_1)\cap\Cone(D_2)=\Cone(D_3)$.
		\item Conversely, if $D_3$ is a finest common coarsening of $D_1$ and $D_2$ among all Delaunay decompositions, then $\Cone(D_1)\cap\Cone(D_2)=\Cone(D_3)$.
	\end{enumerate}

\end{corollary}

The number of Delaunay decompositions of a fixed punctured Riemann
surface $\Riem$ with variable decoration, hence the number of
secondary cones, is finite~\cite{akiyoshi}, see also \cite{luo}.  We
arrive at a central object of our study.

\begin{thmdef}
  \label{def:sfan}
  The collection of secondary cones of Delaunay decompositions,   
  \begin{equation}
    \secfan(\Riem) \ = \ \{~ \Cone(\DD) ~|~ \DD \text{ Delaunay
      decomposition of } \Riem ~\} \enspace.\label{eq:sfan}
  \end{equation}
  is a finite polyhedral fan with support $\Rpos\setminus\{0\}$, called the \emph{secondary fan} of the punctured Riemann surface $\Riem$.
\end{thmdef}
	
The next examples should be compared with \cite[\S3]{TillmannWong:2016}.

\begin{example}[Once punctured torus with symmetric metric]\label{ex:T1}
Consider the triangulation $\DD$ of the once punctured torus depicted in Figure~\ref{fig:ex_T1}.
We equip this torus with a decorated hyperbolic structure by choosing
$\lambda$-lengths 3 and 4 for the outer edges and 5 for the
diagonal. Note that the $\lambda$-length of the flipped diagonal is
also $5$ by Ptolemy's relation. As always with once-punctured
surfaces, there is a unique Delaunay decomposition independent of the
choice of horocycle. In this case it is the decomposition obtained by
omitting either diagonal in Figure~\ref{fig:ex_T1}. It has one face,
which is a quadrilateral.
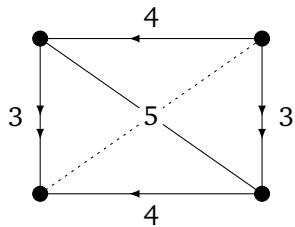
\begin{figure}[tbh]
	\begin{tikzpicture}[scale=1.2,mydot_out/.style={circle,fill=black,draw,outer sep=0pt,inner sep=2pt}, mydot_in/.style={circle,fill=white,draw,outer sep=0pt,inner sep=2pt}]
	\oncepuncturedtorus
	\node at (0,1.1){$4$};
	\node at (0,-1.1){$4$};
	\node at (1.5,0){$3$};
	\node at (-1.5,0){$3$};
	\draw node[circle,fill=white,inner sep=0.6pt] {$5$};
\end{tikzpicture}
	\caption{The triangulation $\DD$ of the once punctured torus from Example~\ref{ex:T1}. The dashed edge is obtained by flipping the diagonal.}
	\label{fig:ex_T1}
\end{figure}
 
\end{example}

\begin{example}[Twice punctured torus $\Riem_{1,2}$ with symmetric metric]\label{ex:T2}
  A twice punctured torus can be represented by a fundamental hexagon
  with opposite edges identified.  Let $\DD$ be the triangulation
  illustrated in the lower right of Figure~\ref{fig:ex_T2}, where the
  first puncture is black and the second one white.  Consider the
  decorated hyperbolic structure that is obtained by setting all six
  $\lambda$-lengths to $1$.  We denote the resulting hyperbolic torus
  with two cusps by $\Riem_{1,2}$.  By equation~\eqref{eq:h-lengths}
  all $h$-lengths are equal to $1$ as well, yielding $\omega_1=9$ and
  $\omega_2=3$ via equation~\eqref{eq:omega_sum}.  Since the local
  Delaunay conditions~\eqref{eq:local_Del} hold for all edges with
  strict inequality, it follows that $\DD=\DD(9,3)$.  The
  inequalities~\eqref{eq:edge_ineq} defining the secondary cone
  $C(\DD)$ become
  \begin{align*}
    \omega_1 - \omega_2 & \geq 0 \qquad \text{for the black/black edges and}\\
    \omega_2 & \geq 0 \qquad \text{for the black/white edges}.
  \end{align*}
  Hence $\Cone(\DD)$ is spanned by $(1,0)$ and $(1,1)$. 
  By flipping the three black/black edges one after the other and then the first one again, one obtains the triangulation $\DD'$.
  Using the Ptolomy relations~\eqref{eq:ptolomy}, we see that all white/white edges in $\DD'$ have $\lambda$-length $3$ with respect to weights $(9,3)$. 
  For $\DD'$, the inequalities~\eqref{eq:edge_ineq} become
  \[
  \begin{aligned}
    -\omega_1 + \omega_2 & \geq 0 \qquad \text{for the white/white edges and}\\ 
    \omega_1 & \geq 0 \qquad \text{for the black/white edges}.  
  \end{aligned}
  \]     
  Hence $\Cone(\DD^\prime)$ is spanned by $(0,1)$ and $(1,1)$, so
  $\DD'=\DD(3,9)$, for example.
  The Delaunay decomposition that corresponds to the one-dimensional cone $\R_{>0}\,(1,1)$ is obtained by omitting those edges of $\DD$ that are weakly Delaunay with respect to weights $(1,1)$.
  These are precisely the black/black edges.
  Equivalently, we could have started with $\DD'$ and omitted the white/white edges. 
  The Delaunay decomposition $\DD(1,1)$ is thus a decomposition into one ideal hexagon.
  If $\omega_2=0$, the black/white edges of $\DD$ are weakly Delaunay.
  Hence, the Delaunay decomposition $\DD(1,0)$ is a decomposition into an ideal triangle and a punctured ideal triangle.
  The case of $\omega_1=0$ is anlogous.
\end{example}

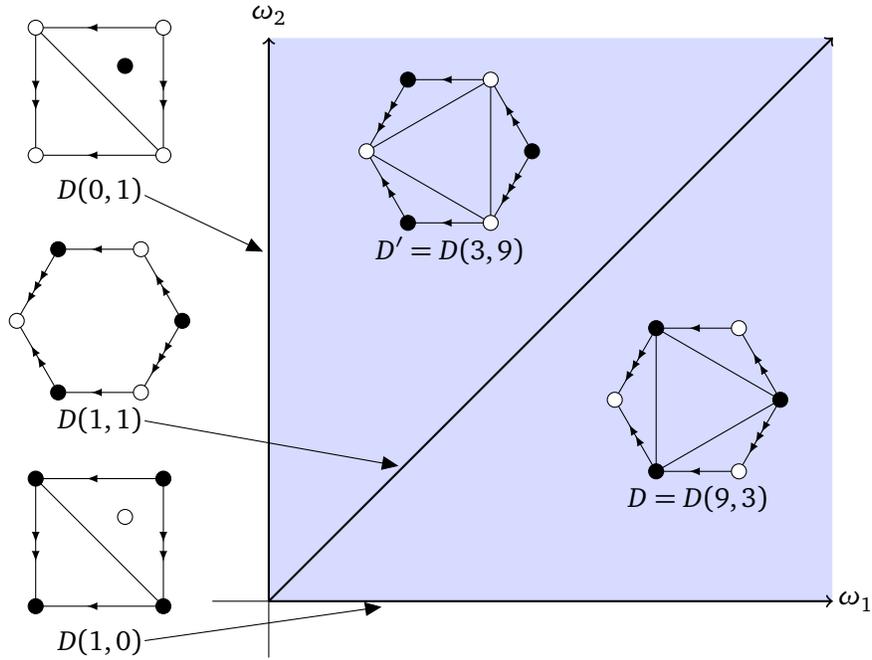
\begin{figure}[tbh]
  \begin{tikzpicture}[scale=1.5]

\draw[color=white,fill=blue2!20] (0,0) -- (5,5) -- (5,5) -- (5,0) -- cycle;
\draw[color=white,fill=blue2!20] (0,0) -- (5,5) -- (0,5) -- cycle;
\draw[<-,thick] (0,5) -- (0,0); \draw (0,0)-- (0,-0.5);	
\draw (-0.5,0)-- (0,0); \draw[<-,thick] (5,0) -- (0,0);
\draw[thick,->] (0,0) -- (5,5);
\node at (0,5.2) {$\omega_2$};
\node at (5.2,0) {$\omega_1$}; 

\node at (3.8,1.6) {
	\begin{tikzpicture}[scale=1.1]
	\hexagon
	\node at (0,-1.2) {$\DD = \DD(9,3)$};
	\end{tikzpicture}
	};

\node at (1.6,3.8) {
	\begin{tikzpicture}[scale=1.1]
	\hexagontwo 
	\node at (0,-1.2) {$\DD^\prime=\DD(3,9)$};
	\end{tikzpicture}
	};

\node at (-1.5,2.3) {
	\begin{tikzpicture}[scale=1.1]
	\hexagonempty
	\node at (0,-1.2) {$\DD(1,1)$};
	\end{tikzpicture}
	};
	
\node at (-1.5,0.3) {
	\begin{tikzpicture}[scale=1.2,mydot_out/.style={circle,fill=black,draw,outer sep=0pt,inner sep=2pt}, mydot_in/.style={circle,fill=white,draw,outer sep=0pt,inner sep=2pt}]
	\squarewithpuncture
	\node at (0,-1.1) {$\DD(1,0)$};
	\end{tikzpicture}
};

\node at (-1.5,4.3) {
	\begin{tikzpicture}[scale=1.2,mydot_out/.style={circle,fill=white,draw,outer sep=0pt,inner sep=2pt}, 	mydot_in/.style={circle,fill=black,draw,outer sep=0pt,inner sep=2pt}]
	\squarewithpuncture
	\node at (0,-1.1) {$\DD(0,1)$};
	\end{tikzpicture}
};	

\draw[->,>=triangle 45] (-1.1,3.6) -- (-0.05,3.1);
\draw[->,>=triangle 45] (-1.1,1.6) -- (1.15,1.2);
\draw[->,>=triangle 45] (-1.1,-0.35) -- (1,-0.05);

\end{tikzpicture}
  \caption{Secondary fan of the twice punctures torus $\Riem_{1,2}$ from Example~\ref{ex:T2} and its Delaunay decomposition.}
  \label{fig:ex_T2}
\end{figure}

\begin{example}[Twice punctured torus $\Riem'_{1,2}$ with generic metric] \label{ex:T2b}	
  Consider the decorated structure on the twice punctured torus $\DD$ from Example~\ref{ex:T2} that corresponds to $\lambda$-lengths $(1, \tfrac{3}{2}, 1, \tfrac{1}{2}, 1, \tfrac{3}{4})$, where we ordered the edges as indicated in Figure~\ref{fig:T2b}.
  We skip the calculations and present the secondary fan in Figure~\ref{fig:T2b}.
  It has nine maximal cones, each of which is associated to a triangulation.
  Adjacent maximal cones correspond to triangulations that differ in a flip.  
  Note that the triangulations $\DD$ and $\DD'$ occur as Delaunay triangulations, but this time in between there also are three triangulations that describe a flip path between $\DD$ and $\DD'$.
  Furthermore, when moving to the boundary of the secondary fan, e.g., with $\omega_1$ tending to $0$, the three black/white edges do not vanish simultaneously at $\omega_1=0$, as in Example~\ref{ex:T2}.
  Instead, two of them get flipped to become black/black edges.
  The two triangulations at the boundary of the secondary fan then contain only one black/white edge, appearing as the self folded edge of one triangle.
  The cases of $\omega_i=0$ with $i=1,2$ then correspond to decompositions of the torus into three ideal triangles and one punctured monogon.    
\end{example}
Experimental evidence suggests that the secondary fan of the twice punctured torus always has precisely nine maximal cones, for any choice of $\lambda$-lengths which is generic.

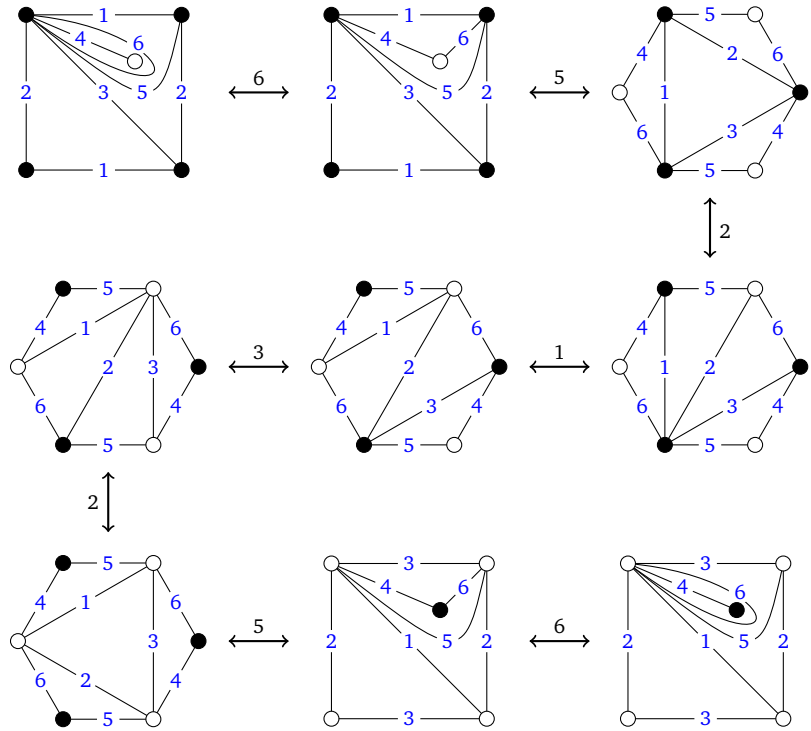
\begin{figure}[tbh]
	\begin{tikzpicture}
\node at (-4,3.65) {
		\begin{tikzpicture}[scale=1.46,mydot_out/.style={circle,fill=black,draw,outer sep=0pt,inner sep=2pt}, mydot_in/.style={circle,fill=white,draw,outer sep=0pt,inner sep=2pt}]
			\squarethreewithlabel{1}{2}{3}
		\end{tikzpicture}
};
\draw[thick,<->] (-2.4,3.65) -- (-1.6,3.65); \node at (-2,3.85) {\scriptsize $6$};
\node at (0,3.65) {
		\begin{tikzpicture}[scale=1.46,mydot_out/.style={circle,fill=black,draw,outer sep=0pt,inner sep=2pt}, mydot_in/.style={circle,fill=white,draw,outer sep=0pt,inner sep=2pt}]
			\squaretwowithlabel{1}{2}{3}
		\end{tikzpicture}
};
\draw[thick,<->] (1.6,3.65) -- (2.4,3.65); \node at (2,3.85) {\scriptsize $5$};
\node at (4,3.65) {
	\begin{tikzpicture}[scale=1.2]
			\hexagonwithlabel 
	\end{tikzpicture}
};
\draw[thick,<->] (4,2.25) -- (4,1.45); \node at (4.2,1.8) {\scriptsize $2$};
\node at (-4,0) {
		\begin{tikzpicture}[scale=1.2]
			\hexagonwithlabelthree
		\end{tikzpicture}
};
\draw[thick,<->] (-2.4,0) -- (-1.6,0); \node at (-2,0.2) {\scriptsize $3$};
\node at (0,0) {
		\begin{tikzpicture}[scale=1.2]
				\hexagonwithlabeltwo
		\end{tikzpicture}
};
\draw[thick,<->] (1.6,0) -- (2.4,0); \node at (2,0.2) {\scriptsize $1$};
\node at (4,0) {
		\begin{tikzpicture}[scale=1.2]
			\hexagonwithlabelone
		\end{tikzpicture}
};
\draw[thick,<->] (-4,-2.2) -- (-4,-1.4); \node at (-4.2,-1.8) {\scriptsize $2$};
\node at (-4,-3.65) {
		\begin{tikzpicture}[scale=1.2]
			\hexagonwithlabelfour
		\end{tikzpicture}
};
\draw[thick,<->] (-2.4,-3.65) -- (-1.6,-3.65); \node at (-2,-3.45) {\scriptsize $5$};
\node at (0,-3.65) {
	\begin{tikzpicture}[scale=1.46,mydot_out/.style={circle,fill=white,draw,outer sep=0pt,inner sep=2pt}, mydot_in/.style={circle,fill=black,draw,outer sep=0pt,inner sep=2pt}]
				\squaretwowithlabel{3}{2}{1}
			\end{tikzpicture}
};
\draw[thick,<->] (1.6,-3.65) -- (2.4,-3.65); \node at (2,-3.45) {\scriptsize $6$};
\node at (4,-3.65) {
		\begin{tikzpicture}[scale=1.46,mydot_out/.style={circle,fill=white,draw,outer sep=0pt,inner sep=2pt}, mydot_in/.style={circle,fill=black,draw,outer sep=0pt,inner sep=2pt}]
			\squarethreewithlabel{3}{2}{1}
		\end{tikzpicture}
};	
\end{tikzpicture}
	\caption{The flip sequence of Delaunay triangulations of the twice punctured torus $\Riem'_{1,2}$ from Example~\ref{ex:T2b}.} 
	\label{fig:T2b}
\end{figure}

\begin{example}[Sphere with three punctures]\label{ex:S3}
	The sphere with three punctures is special in two ways.
	First, its Teichm\"uller space is a point.
	We denote the unique hyperbolic sphere with three cusps by $\Riem_{0,3}$. 
	The decorated Teichm\"uller space becomes the space of decorations on $\Riem_{0,3}$.
	The second special feature of the sphere with three punctures is that is admits exactly four triangulations, each of which turns out to be Delaunay.
	The secondary fan of $\Riem_{0,3}$ is depicted in Figure~\ref{fig:fanS3}.
	The central top-dimensional secondary cone $\Cone(\triang)$ is given by the inequalities 
	\begin{equation*}
		\begin{split}
		\omega_1 + \omega_2 -\omega_3 \ &\geq\ 0 \qquad (\text{for edge $a$})\\
		\omega_1 - \omega_2 + \omega_3 \ &\geq\ 0 \qquad (\text{for edge $c$})\\ 
		-\omega_1 + \omega_2 + \omega_3 \ &\geq\ 0 \qquad (\text{for edge $b$}) \enspace .		
		\end{split}
	\end{equation*}
	The other three top-dimensional secondary cones correspond to the remaining three triangulations, each of which is obtained from $T$ by flipping either of the edges $a,b$ or $c$, introducing edges $a',b'$ or $c'$, respectively. 
	\begin{figure}[tbh]
		\input{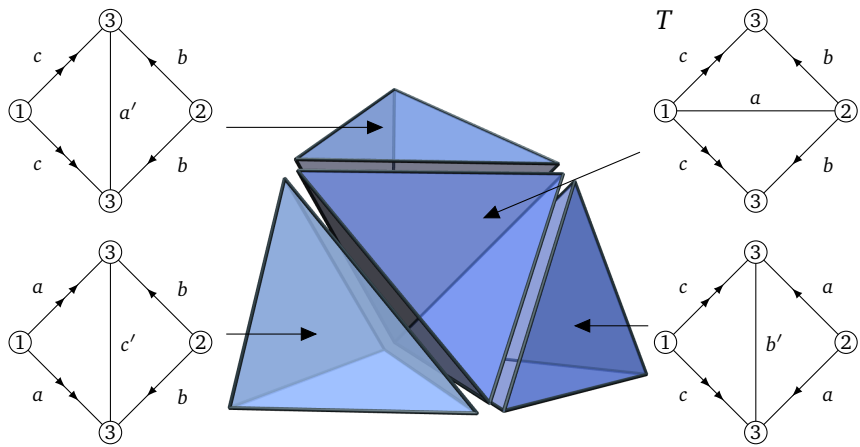}
		\caption{Secondary fan of the sphere with three punctures.}
		\label{fig:fanS3}
	\end{figure}	
\end{example}

\begin{example}[Torus with three punctures]\label{ex:T3}
	A triangulated torus with three punctures is depicted in the top left of Figure~\ref{fig:exT3}.
	Let $\Riem_{1,3}$ be the torus with three punctures that is obtained by setting all $\lambda$-lengths in the triangulation $\triang$ to $1$.
	The secondary fan $\secfan(\Riem_{1,3})$ is illustrated in Figure~\ref{fig:exT3}.
	The secondary cone $\Cone(T)$ is the central cone with four rays. The adjacent top-dimensional secondary cone above corresponds to the triangulation that is obtained from $T$ by flipping two edges.
	Furthermore, there is a top-dimensional cone that corresponds to a Delaunay decomposition which is not a triangulation, see bottom right of Figure~\ref{fig:exT3}.
	\begin{figure}[tbh]
			\input{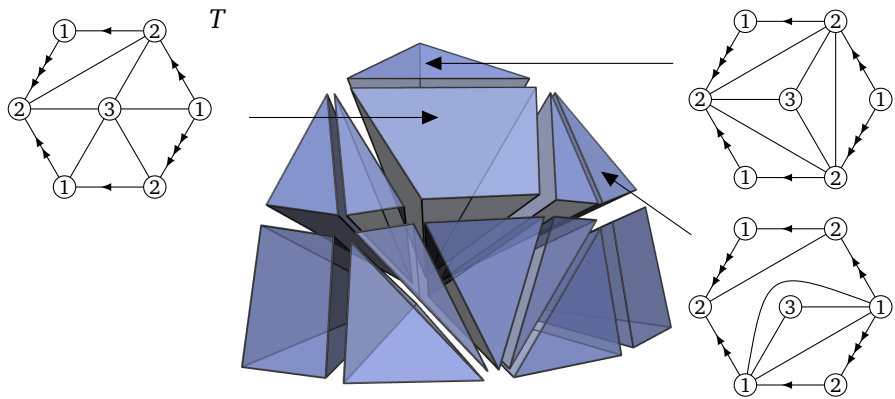}
			\caption{Secondary fan of the torus with three punctures $\Riem_{1,3}$ from Example~\ref{ex:T3}.} 
			\label{fig:exT3}
	\end{figure}
	
\end{example}

Examples~\ref{ex:T1}, \ref{ex:T2} and \ref{ex:T3} illustrate phenomena
that do not occur in the classical theory of secondary fans of point
configurations (cf.\ Section~\ref{sec:gkz}). In
Examples~\ref{ex:T1} and~\ref{ex:T3}, there are top-dimensional cones
which correspond to a Delaunay decomposition that is not a
triangulation.  In Example~\ref{ex:T2}~and~\ref{ex:T3}, there are
adjacent top-dimensional cones which correspond to triangulations that
differ by more than one flip.  In the classical theory of secondary
fans, top-dimensional cones correspond to triangulations and adjacent
top-dimensional cones correspond to triangulations that differ by a
single flip.

Both phenomena also do not occur in Penner's decomposition of the
whole decorated Teichm{\"u}ller space $\decTeich$. There, the
top-dimensional cells correspond to ideal triangulations, and cells of
codimension $k$ correspond to ideal decompositions with $k$ ideal arcs
omitted. Here, we consider the intersection of Penner's decomposition
with one fiber of $\decTeich$. Then the decomposition becomes a
polyhedral fan in the space of horocyclic lengths, but the above
statements are no longer valid.

\section{Secondary Polyhedra}\label{sec:poly}
\noindent
The secondary fan of a punctured Riemann surface $\Riem$ is the normal
fan of a secondary polyhedron, which we will construct in this
section. In fact, we will construct a family of secondary polyhedra
for $\Riem$, one for each point $x\in\Riem$.

The main idea is to transfer the Epstein--Penner convex hull
construction from the hyperboloid model to the hemisphere model via
the projective transformation~\eqref{eq:trafo} (cf.\
Figure~\ref{fig:domes}). Let $\C_\hemi$ denote the image of the convex
hull $\C$ (cf.\ equation~\eqref{eq:Chull}). The vertices of $\C_\hemi$
lie in the light cylinder $\Lplus_\hemi$. 

The group $\Gamma$ acts by projective transformations on both the
hemisphere $\hemi$ and its equatorial disk, which is a Klein disk
$\Klein$ translated to height $0$. Thus, the Riemann surface
$\Riem=\HH/\Gamma$ appears also as $\hemi/\Gamma$ and as
$\Klein/\Gamma$. The Delaunay decompositions of $\hemi/\Gamma$ and
$\Klein/\Gamma$ are obtained by projecting the faces of $\C_\hemi$
vertically to $\hemi$ and $\Klein$, respectively. 

The hemisphere and the equatorial Klein disk share the unit circle
$S^{1}$ at height~$0$ as their circle at infinity. The $i$th cusp of
$\Riem$ corresponds to a $\Gamma$-orbit $\PP_{i}\subseteq S^{1}$. A
horocycle at the $i$ cusp corresponds to a $\Gamma$-orbit
$\tilde{\B}_{i}\subseteq \Lplus_{\hemi}$, which is the image of
$\B_{i}$ under the transformation~\eqref{eq:trafo} and contains
exactly one point vertically above each point in $\PP_{i}$. The height
of a point in $\tilde{\B}_{i}$ depends only on the length $\omega_{i}$
of the corresponding horocycle in $\Riem$ and on the point
$p\in\PP_{i}$ above which it lies. 

\begin{lemma}\label{lemma:linearheights}
  The height $z_p(\omega_i)$ of the point in
  $\tilde{\B}_{i}$ above $p\in\PP_{i}$ depends linearly on
  $\omega_{i}$, i.e.,
  \begin{equation*}
    z_p(\omega_i) \ = \ \omega_i\,z_p(1) \enspace.
  \end{equation*}
\end{lemma}

\begin{proof}
  The point $p$ and the weight $1$ determine a unique point
  $v=(v_1,v_2,v_3)$ on the light cone $\Lplus\subseteq\R^{2,1}$.
  It follows from equation~\eqref{eq:lambda_lightcone} together with equation~\eqref{eq:h-lengths}, that scaling $v$ (and all the points in the $\Gamma$-orbit of $v$) by $\tfrac{1}{\omega_i}$ results in scaling the $h$-lengths at the $i$th cusp by $\omega_i$.
  Thus, by equation~\eqref{eq:omega_sum}, modifying the weight from $1$ to $\omega_i$ changes the point $v$ to $\tfrac{1}{\omega_i} v$.
  If we go from the hyperboloid model to the hemisphere model via the transformation~\eqref{eq:trafo}, we arrive at the point
  $
  \frac{1}{v_3} \, (v_1,v_2,\omega_{i}) \enspace . 
  $
  Hence we obtain $z_p(\omega_i) = \tfrac{\omega_i}{v_3} = \omega_i z_p(1)$.
\end{proof}

Let $\triang$ be an ideal triangulation of $\Riem$ and let
$\hat{\triang}$ be the corresponding $\Gamma$-invariant triangulation
of the Klein disk at height $0$. An ideal triangle $\Delta$ of
$\hat{\triang}$ is also a Euclidean triangle inscribed in the unit
circle. Let $\vol(\Delta)$ denote its Euclidean area. If the vertices
$p$, $q$, $r$ of $\Delta$ correspond to the cups $i$, $j$, $k$,
respectively, then the Euclidean volume of the truncated prism over
$\Delta$ with heights
$\big( z_p(\omega_{i}), z_q(\omega_{j}), z_r(\omega_{k}) \big)$ is
\begin{equation*}
  z_\Delta(\omega) \ := \ 
  \frac{1}{3} \vol (\Delta) \, 
  \bigl( z_p(\omega_{i}) + z_q(\omega_{j}) +
  z_r(\omega_{k}) 
  \bigr) \enspace .
\end{equation*}
We call the polyhedral complex generated by these skew prisms, formed from all ideal triangles in $\hat{\triang}$, the \emph{$\omega$-dome} of $\triang$; see Figure~\ref{fig:domes} (right) for an illustration.
The $\omega$-dome depends not only on $\Riem$, $\triang$ and
$\omega$, but also on the group $\Gamma$. A conjugate subgroup
$A\Gamma A^{-1}$ with $A\in\SO$ has the same Riemann surface as
quotient, but the $\omega$-dome is transformed by the projective
action of $A$. Unless this is a rotation around the vertical axis, the
volume of the transformed $\omega$-dome will in general be different. 
As a consequence, the volume depends on the point $x\in\Riem$ that
corresponds to the origin in the Klein model. We denote the volume of
the $\omega$-dome by $L_{x,\triang}(\omega)$ and obtain
\begin{equation*}
  L_{x,\triang}(\omega) \ = \ \sum_{\Delta\in\hat{\triang}} z_\Delta(\omega) \enspace .
\end{equation*}
Note that this infinite sum has a finite value because it it is the
limit of an increasing and bounded sequence.  The boundedness follows
from the fact that the union of orbits $\B\subseteq\Lplus$ is bounded
away from the origin~\cite{epstein_penner}, which implies that the
image $\tilde{\B}\subseteq\Lplus_\hemi$ is bounded in height.

\begin{lemma}\label{lemma:sumvolumes}
  The function $L_{x,\triang} \colon \Rpos^n \rightarrow \Rpos$ is linear.
\end{lemma}
 
\begin{proof}
  For a vertex $p$ of a triangle $\Delta \in \hat{\triang}$ let
  $\pi(p)$ denote the cusp corresponding to the $\Gamma$-orbit of $p$,
  i.e., $\pi(p)=i$ if $p\in\PP_{i}$. We write $m( \Delta , i )$ for
  the \emph{multiplicity} of the $i$th cusp in the triangle $\Delta$,
  i.e., the number of vertices $p$ of $\Delta$ with $\pi(p)=i$. This
  is a number between zero and three. By rearranging the absolutely
  convergent sum and invoking Lemma~\ref{lemma:linearheights} we
  obtain
  \begin{equation*}
    \begin{split}
      L_{x,\triang}(\omega) \ &= \ \sum_{\Delta\in{\hat{\triang}}} z_\Delta(\omega) \ = \ \frac{1}{3} \sum_{\Delta \in {\hat{\triang}}} \vol (\Delta) \sum_{p \text{ vertex of } \Delta} \omega_{\pi (p)} z_p(1)  \\
      &= \ \frac{1}{3} \sum_{i=1}^n \omega_i \sum_{ \Delta \in {\hat{\triang}} } m( \Delta , i ) ~ \vol (\Delta) ~ \sum_{p \text{ vertex of } \Delta} z_p(1) \\
      &= \ \frac{1}{3} \left\langle \omega , \; \sum_{i=1}^n  \left( \sum_{\Delta \in {\hat{\triang}}} m( \Delta , i ) ~ \vol (\Delta) ~ \sum_{p \text{ vertex of } \Delta} z_p(1) \right) e_i \right\rangle \enspace.
    \end{split}
  \end{equation*}
\end{proof}

\begin{definition}\label{def:gkz}
  We call
  \[ \phi_{x,\triang} \ = \ \sum_{i=1}^n  \left( \sum_{\Delta \in {\hat{\triang}}} m ( \Delta , i ) ~ \vol (\Delta) ~ \sum_{p \text{ vertex of } \Delta} z_p(1) \right) e_i  \]
  the \emph{GKZ vector} of the ideal triangulation $\triang$. 
\end{definition}
By definition we have $ L_{x,T}(\omega) = \tfrac{1}{3}\scalp{ \omega , \phi_{x,T} }$.
The following is our hyperbolic analog of a key characterization of regular triangulations in the Euclidean setting; see \cite[Lemma 5.2.14]{bibel}.
\begin{proposition}[Hyperbolic Crucial Lifting Lemma]\label{prop:crucial}
  Let $\triang$ be a triangulation of $\Riem$, $x\in\Riem$ and let $\omega \in \Rpos^n \setminus \{0\}$ be an arbitrary weight vector.
  Then the following are equivalent: 
  \begin{enumerate}[label = (\roman*)]
  \item The triangulation $\triang$ refines the Delaunay decomposition $\DD(\omega)$.
  \item We have $L_{x,\triang}(\omega) \geq L_{x,\triang'}(\omega)$ for all triangulations $\triang'$ of $\Riem$, with equality if and only if $\triang'$ refines $\DD(\omega)$ as well.
  \end{enumerate} 
\end{proposition}

\begin{proof}
  The Delaunay decomposition $\DD(\omega)$ in the Klein disk is
  obtained by orthogonally projecting the upper boundary of the convex hull $\C_\hemi(\omega)$ to the Klein disk.
  Thus $\triang$ refines  $\DD(\omega)$ if and only if the $\omega$-dome of $\triang$ equals $\C_\hemi(\omega)$.
	
  Let $\triang'$ be an arbitrary triangulation of $\Riem$ and let $\Delta$ be some triangle in $\triang'$.
  The skew prism over $\Delta$ with heights prescribed by $\omega$ is contained in $\C_\hemi(\omega)$ by convexity. 
  Furthermore, its top triangle is contained in a face of $\C_\hemi(\omega)$ if and only if $\Delta$ refines a cell of $\DD(\omega)$.
  It follows that the volume of the $\omega$-dome of $\triang'$ does not exceed the volume of $\C_\hemi(\omega)$.
  These two volumes coincide if and only if each triangle of $\triang'$ refines a cell of $\DD(\omega)$.	  
\end{proof}

\begin{figure}[tbh]
  \begin{minipage}{0.45\textwidth}
    \centering\includegraphics[width=0.9\linewidth]{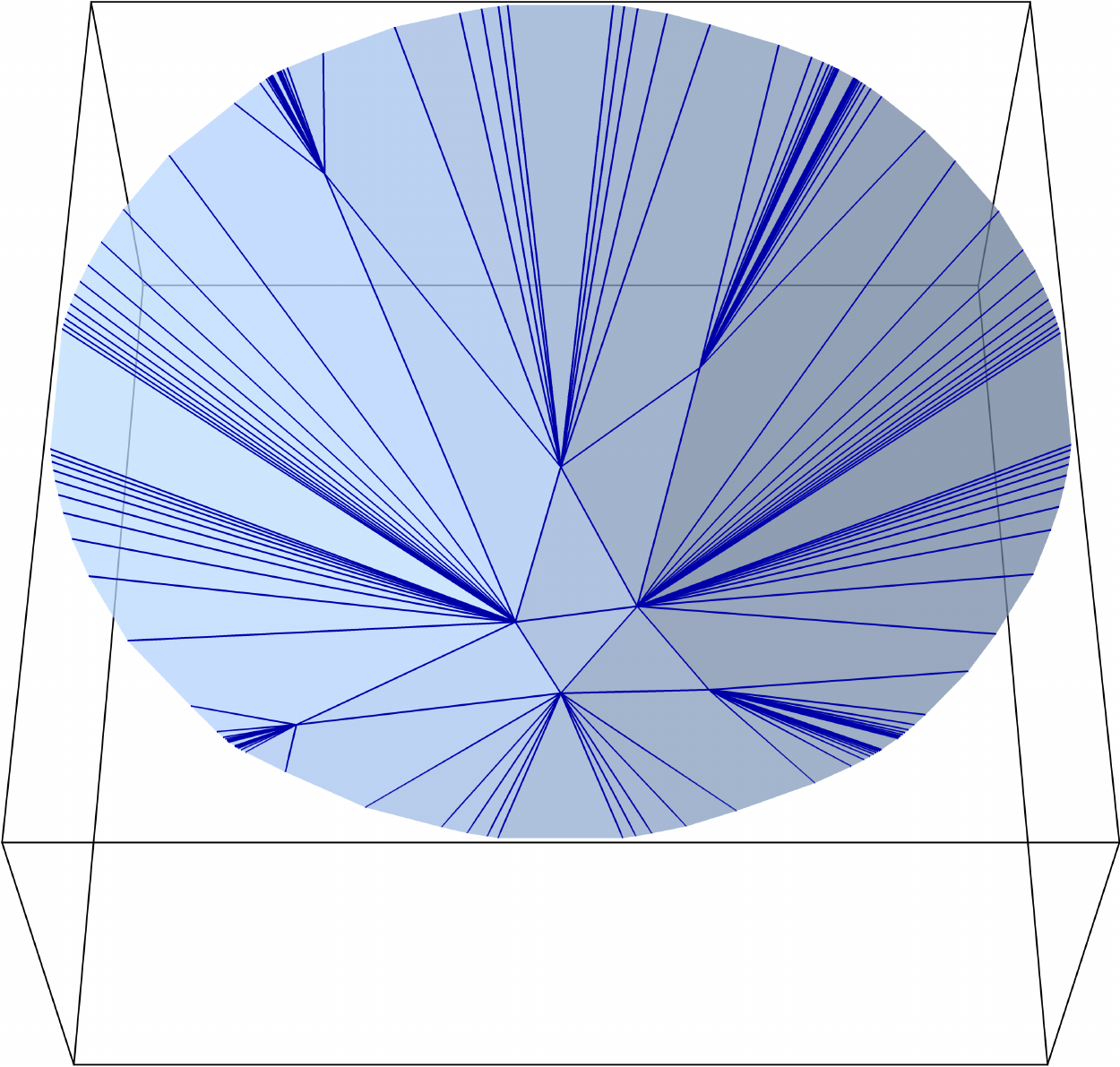}
  \end{minipage}
  \quad
  \begin{minipage}{0.45\textwidth}
    \centering\includegraphics[width=0.9\linewidth]{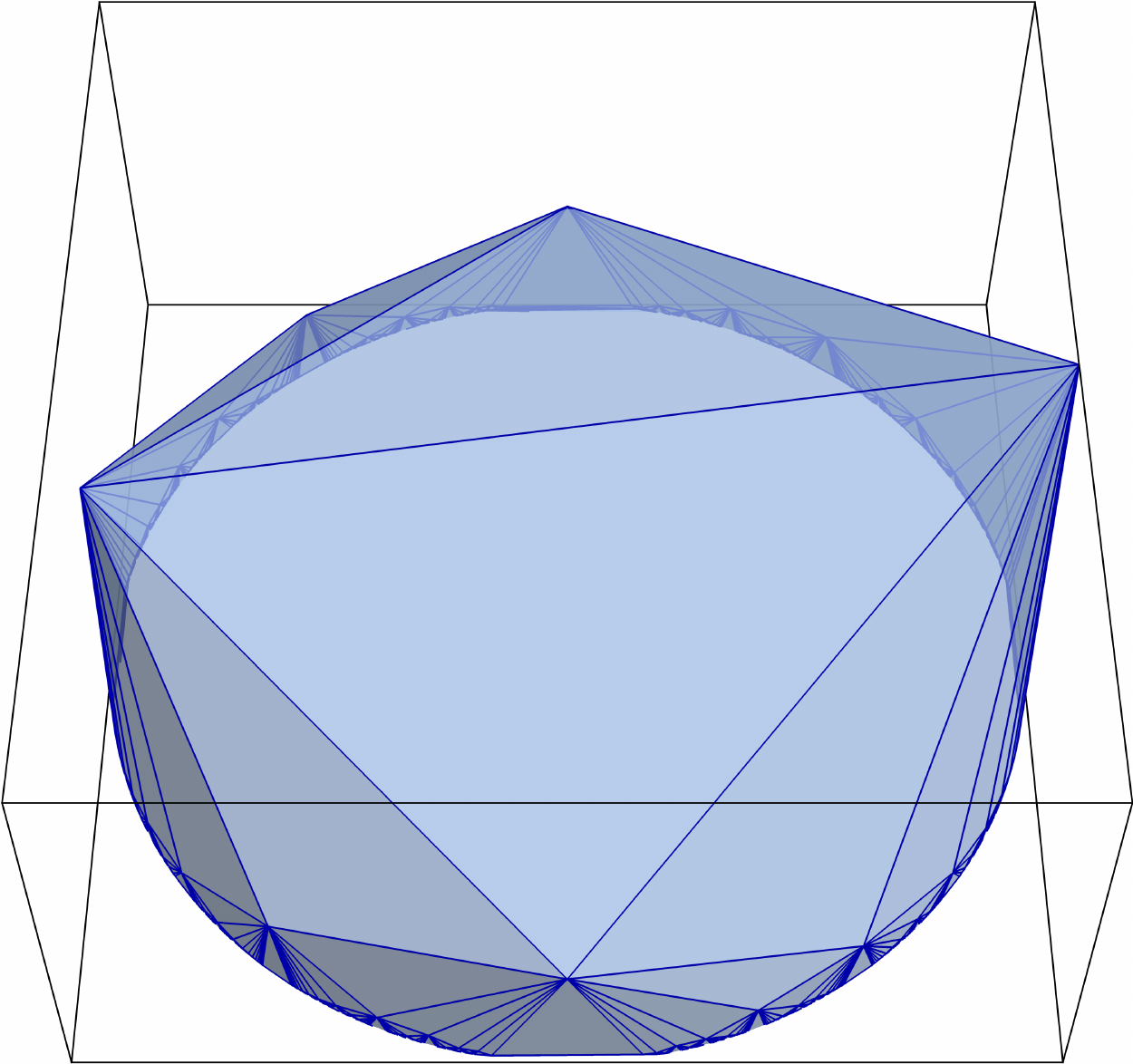}
  \end{minipage}
\caption{Convex hull construction in the hyperboloid model (left) and
  in the hemisphere model (right)}
  \label{fig:domes}
\end{figure}

\begin{definition}\label{def:secpoly}
  For a point $x\in\Riem$ the \emph{secondary polyhedron} of ($\Riem,x$) is defined as
  \begin{equation}
    \label{eq:secpoly}
    \secpoly(\Riem,x) \ = \ \conv \{ ~\phi_{x,T}~|~T \text{ triangulation of } \Riem~ \} ~+~ \R_{\leq0}^n \enspace.
  \end{equation}
\end{definition}


\begin{theorem}
  \label{thm:normalfan}
	For each $x\in\Riem$, the secondary fan $\secfan(\Riem)$ is the outer normal fan of the secondary polyhedron $\secpoly(\Riem,x)$.
\end{theorem}

\begin{proof}
  Let $T$ be a triangulation of $\Riem$.  Denote by
  $N_{\secpoly}(\phi_{x,T})$ the the outer normal cone of the GKZ
  vector $\phi_{x,T}$ in $\secpoly(\Riem,x)$.  It suffices to show that
  \begin{equation*}
    N_{\secpoly}(\phi_{x,T}) = \Cone(T) \enspace.
  \end{equation*}
  First, let $\omega \in \Cone(T)$.  By definition
  $T \preceq D(\omega)$.  Since
  $L_{x,T}(\omega)=\frac{1}{3}\scalp{\omega , \phi_{x,T}}$, it follows from
  Proposition~\ref{prop:crucial} that
  \begin{equation*}
    \scalp{\omega , \phi_{x,T}} \geq \scalp{\omega , \phi_{x,T^\prime}}
  \end{equation*}
  for all triangulations $T^\prime$ of $\Riem$. This shows that
  $ N_{\secpoly}(\phi_T) \supseteq \Cone(\Riem, T). $ The other
  inclusion follows similarly.
\end{proof}

\begin{example}\label{ex:T21poly}
	We continue the case of the twice punctured torus $\Riem_{1,2}$ from Example~\ref{ex:T2}.
	The vertices of the secondary polyhedron $\secpoly(\Riem_{1,2})$ are the GKZ vectors of the two Delaunay triangulations $\DD$ and $\DD'$, approximately given by
	\[ \phi_\DD = (2.15606 ,\; 0.63724)~, \quad \phi_{\DD'} = (0.58118 ,\; 2.21212) \enspace .\]
	Note that both $\phi_\DD$ and $\phi_{\DD'}$ lie on the hyperplane defined by $\omega_1 + \omega_2 = 2.7933$.
	Thus, the bounded edge of that connects the two vertices has outer normal vector $(1,1)$.
	It follows that the outer normal fan of $\secpoly(\Riem_{1,2})$ equals $\secfan(\Riem_{1,2})$, indeed.
	Now consider the triangulation $\triang$ obtained by flipping one black/black edge in $\DD$.
	The GKZ vector of that triangulation is $\phi_\triang = ( 1.8622,\; 0.9311)$ and it lies in the relative interior of the bounded edge. 
	Since $\triang$ refines the Delaunay decomposition $\DD(1,1)$, this illustrates the case of equality in Proposition~\ref{prop:crucial}.
	By further flipping a black/white edge of $\triang$, we get a triangulation $\triang'$ that does not refine any Delaunay decomposition.
	Note that we change the fundamental hexagon in order to illustrate this triangulation in Figure~\ref{fig:gkz_vertices}.
	The corresponding GKZ vector is $\phi_{\triang'} = (1.79904,\; 0.55685)$ and lies in the interior of the secondary polyhedron.
	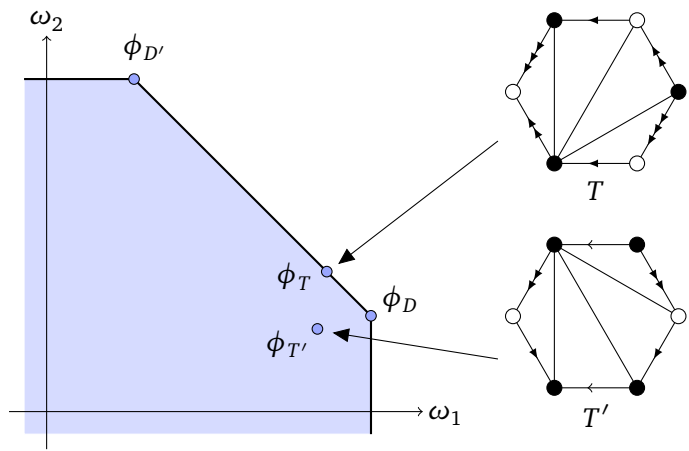
\begin{figure}[tbh]
	  \begin{tikzpicture}[scale=1]

\draw[color=white,fill=blue2!20] (-0.3,-0.3) -- (-0.3,2*2.21212) -- (2*0.58118,2*2.21212) -- (4.31212,1.27448) -- (4.31212,-0.3) -- cycle;
\draw[thick] (-0.3,2*2.21212) -- (2*0.58118,2*2.21212) -- (4.31212,1.27448) -- (4.31212,-0.3);
\draw[<-] (0,5) -- (0,-0.5);	
\draw[<-] (5,0) -- (-0.5,0);
\node at (0,5.2) {$\omega_2$};
\node at (5.3,-0.05) {$\omega_1$}; 
\filldraw[fill=blue2!50] (4.31212,1.27448) circle (2pt);
\filldraw[fill=blue2!50] (2*0.58118,2*2.21212) circle (2pt);
\filldraw[fill=blue2!50] (2*1.8622,2*0.9311) circle (2pt);
\filldraw[fill=blue2!50] (2*1.8, 1.1) circle (2pt);
\node at (4.7,1.5) {$\phi_\DD$};
\node at (1.3,4.85) {$\phi_{\DD'}$};
\node at (3.3,1.8) {$\phi_\triang$};
\node at (3.2,0.9) {$\phi_{\triang'}$};
\draw[->,>=triangle 45] (6,3.6) -- (2*1.8622+0.15,2*0.9311+0.1);
\draw[->,>=triangle 45] (6,0.7) -- (3.8,1.05);

\node at (7.3,4) {
	\begin{tikzpicture}[scale=1.1]
	\hexagonflipone
	\node at (0,-1.2) {$\triang$};
	\end{tikzpicture}
	};

\node at (7.3,1) {
	\begin{tikzpicture}[scale=1.1]
	\hexagonfliponefour
	\node at (0,-1.2) {$\triang'$};
	\end{tikzpicture}
	};

\end{tikzpicture}
	  \caption{The secondary polyhedron of the twice punctured torus from Example~\ref{ex:T21poly}.}
	  \label{fig:gkz_vertices}
	\end{figure}
\end{example}

\section{Calculating Secondary Cones, Fans and Polyhedra}
\label{sec:computing}
\noindent
We briefly want to sketch how the examples in the previous section have been computed.
Our method is implemented in \polymake \cite{polymake:2000}.
Throughout the decorated punctured Riemann surface $\Riem$ is given in terms of Penner coordinates, i.e., some topological triangulation $\triang$ of $\Sgn$ together with $\lambda$-lengths of its edges.

The first task is to compute a secondary cone $\Cone(\omega)$ that contains a given weight vector $\omega$ in its relative interior.
Starting out with the triangulation $\triang$ we can employ Week's flip-algorithm \cite{Weeks:1993} to obtain a Delaunay triangulation $\DD$ that refines $\DD(\omega)$ by successively flipping edges which violate the local Delaunay condition (cf.\ Lemma~\ref{lem:local_Del}).
Then the secondary cone $\Cone(\omega)=\Cone(\DD(\omega))$ is
determined by the inequalities \eqref{eq:edge_ineq}. (For a detailed analysis of the flip-algorithm, including a proof that
works also for projective surfaces, see \cite{TillmannWong:2016}. The
case of partially decorated surfaces is discussed in \cite[\S5]{springb_uniform}.)

This allows us to compute the entire secondary fan as follows.
We pick some positive weight vector $\omega$ and compute the secondary cone $\Cone(\omega)$ by the subroutine which we described above.
Generically, $\Cone(\omega)$ is top-dimensional. Otherwise, we pick
another random $\omega$ and try again until $\Cone(\omega)$ is top
dimensional.
%
The goal now is to compute the rest of the secondary fan via a breadth-first search in the dual graph of the secondary fan.
More precisely, we maintain a queue of pairs $(C,F)$, where $C$ is a top-dimensional secondary cone and $F$ is a facet of $C$ which does not lie in the boundary of the positive orthant.
This queue is initialized with $\Cone(\omega)$ and its facets.
The main loop of the algorithm picks a pair $(C,F)$ from the queue while it is not empty.
For a weight vector $\omega'$ with $C'=\Cone(\omega')$ we pick another weight vector $\omega''$ on the line perpendicular to $F$ which contains $\omega'$ such that $C''=\Cone(\omega'')$ and $C'$ are adjacent, i.e., they share $F$ as a common facet.
It may happen that $C''$ is a top-dimensional cone that we saw before.
If, however, the secondary cone $C''$ is new then the pairs formed by $C''$ and its facets other than $F$ are added to the queue.
This algorithm for computing the secondary fan of a punctured Riemann surface is very similar to the method implemented in \gfan \cite{gfan} for the Euclidean setting.

It remains to explain how to obtain the vertices of the secondary polyhedra.
For this we loop through all top-dimensional cones of the secondary fan and compute the GKZ-vector $\phi_\DD$ from some Delaunay triangulation $\DD$ as in Definition~\ref{def:gkz}.
Clearly, that formula does not yield a finite procedure, which is why we need to be content with an approximation of $\phi_D$.

\section{Concluding Remarks and Open Questions}\label{sec:concluding}
\noindent
A standard way of measuring the combinatorial complexity of a convex polytope, or a combinatorial manifold is the \emph{face vector}, or \emph{$f$-vector} for short, which records the number of cells per dimension.
\begin{question}
  What are the possible $f$-vectors of the secondary fans and the secondary polyhedra of punctured Riemann surfaces?
\end{question}
Upper bounds on $f$-vectors often translate into upper bounds on the complexity of related algorithms.
Since the Delaunay triangulations correspond to secondary cones of maximal dimension, answering the previous question would, in particular, imply an upper bound on their number.

The convex hull construction of Epstein and Penner works for 
cusped hyperbolic manifolds of arbitrary dimension. Our construction of secondary fans and
secondary polyhedra also generalizes to the higher dimensional setting
in a straightforward way. Similarly, there is a version of the
GKZ-construction for compact or cusped hyperbolic manifolds with
marked points. This leads to the study of regular triangulations and
decompositions of compact hyperbolic manifolds. Moreover, in dimension
two, one may allow surfaces with cone-like singularities. While all these
generalizations and extensions are fairly straightforward, many
details of the constructions and some phenomena that may be observed are
particular to each setting. 
In the present article we avoided the maximal possible generality in 
favor of a concise presentation of our key ideas.

Another direction for future research involves manifolds with a
geometric structure that is< not metric. For example, Cooper and Long \cite{CooperLong:2015} recently
generalized the Epstein--Penner construction to projective manifolds.
So the following question is natural.
\begin{question}
  Do secondary fans and secondary polyhedra of projective manifolds exist?
\end{question}

Our construction yields a secondary polyhedron for each point
in a punctured Riemann surface.
\begin{question}
  Can anything interesting be said about the dependence of the secondary
  polyhedron $\secpoly(\Riem,x)$ on the point $x\in\Riem$?
\end{question}


\goodbreak

\bibliography{main}
\bibliographystyle{alpha}

\end{document}